\documentclass[oneside,english]{amsart}
\usepackage[T1]{fontenc}
\usepackage[latin9]{inputenc}
\usepackage{amsthm}
\usepackage{amssymb}
\usepackage{esint}

\numberwithin{equation}{section} 
\numberwithin{figure}{section} 
\theoremstyle{plain}
\theoremstyle{plain}
\newtheorem{thm}{Theorem}
  \theoremstyle{plain}
  \newtheorem{prop}[thm]{Proposition}
  \theoremstyle{remark}
  \newtheorem{rem}[thm]{Remark}
  \theoremstyle{plain}
  \newtheorem{lem}[thm]{Lemma}

\theoremstyle{plain}
  \newtheorem{cor}[thm]{Corollary}

  \newtheorem{fact}[thm]{Fact}
  \newcounter{casectr}
  \newenvironment{caseenv}
  {\begin{list}{{\itshape\ Case} \arabic{casectr}.}{%
   \setlength{\leftmargin}{\labelwidth}
   \addtolength{\leftmargin}{\parskip}
   \setlength{\itemindent}{\listparindent}
   \setlength{\itemsep}{\medskipamount}
   \setlength{\topsep}{\itemsep}}
   \setcounter{casectr}{0}
   \usecounter{casectr}}
  {\end{list}}

\usepackage{babel}
\title[Bilipschitz invariants for holomorphic foliations]{Bilipschitz invariants for germs of holomorphic foliations}

\author{Rudy Rosas}

\email{rudy.rosas@pucp.pe}
\email{rudy@imca.edu.pe}

\address{Pontificia Universidad Cat\'olica del Per\'u, Av Universitaria 1801,  San Miguel, Lima, Per\'u.}
\address{Instituto de Matem\'atica y Ciencias Afines, calle Los Bi\'ologos 245, Lima 12, Per\'u.}

\begin{document}

\maketitle
\begin{abstract}In this paper we study bilipschitz equivalences
of germs of holomorphic foliations in $(\mathbb{C}^2,0)$. We prove  that the algebraic multiplicity of a singularity is invariant by such equivalences. Moreover, for a large class of singularities,  we show that the projective holonomy representation is also a bilipschitz invariant.

\end{abstract}

\section{Introduction} \footnotetext{This work was partially supported by the Vicerectorado de Investigacion de la Pontificia Universidad Cat\'olica del Per\'u.}
Given a reduced holomorphic curve $f:(\mathbb{C}^2,0)\rightarrow (\mathbb{C},0)$, singular at
$0\in\mathbb{C}^2$, we define its \emph{algebraic multiplicity} as the
degree of the first   nonzero jet of $f$, that is, $\nu(f)=\nu$
where $$f=f_{\nu}+f_{\nu+1}+\cdots$$ is the Taylor development of
$f$ and $f_{\nu}\neq 0$. A well known result by Burau \cite{burau}
and Zariski \cite{zariski} states that $\nu$ is a
\emph{topological invariant}, that is, given
$\widetilde{f}:(\mathbb{C}^2,0)\rightarrow(\mathbb{C},0)$ reduced and a homeomorphism
$\mathfrak{h}:\mathfrak{U}\rightarrow\widetilde{\mathfrak{U}}$ between neighborhoods of $0\in\mathbb{C}^2$
such that $\mathfrak{h}(f^{-1}(0)\cap {\mathfrak{U}})=\widetilde{f}^{-1}(0)\cap \widetilde{\mathfrak{U}}$ then
$\nu(f)=\nu(\widetilde{f})$. This question in dimension greater than 2 is the celebrated Zariski's Multiplicity Problem, posed in \cite{zariski problem} (see \cite{eyral} for a survey). It is well known that the topological equivalence between the curves $f$ and $\tilde{f}$  implies the topological equivalence between the foliations induced by $df$ and $d\tilde{f}$ on a neighborhood of $0\in\mathbb{C}^2$ (see \cite{king}). Thus it is natural to extend the Zariski's Multiplicity Problem for holomorphic foliations in $(\mathbb{C}^2,0)$, as was posed by J.F.Mattei. Consider a holomorphic vector
field $Z$ in $\mathbb{C}^2$ with a singularity at $0\in\mathbb{C}^2$. If
$$Z=Z_{\nu}+Z_{\nu+1}+\cdots,\quad Z_{\nu}\neq 0$$ we define $\nu=\nu(Z)$ as the
\emph{algebraic multiplicity} of $Z$ at $0\in\mathbb{C}^2$. The vector field $Z$ defines
a holomorphic foliation by curves $\mathcal{F}$ with  isolated
singularity in a neighborhood of $0\in\mathbb{C}^2$ and the algebraic
multiplicity $\nu(Z)$ depends only on the foliation $\mathcal{F}$. Then:
Is $\nu(\mathcal{F})$
a topological invariant of $\mathcal{F}$?.  In \cite{cls}, the
authors give a positive answer if $\mathcal{F}$ is a
\emph{generalized curve}, that is, if the desingularization of
$\mathcal{F}$ does not contain complex saddle-nodes. If
$\mathcal{F}$ is 1-\emph{dicritical}, that is, after a blow up the
exceptional divisor is not invariant by the strict transform of
$\mathcal{F}$, the conjecture is also true: in this case, it is
not difficult to show that the algebraic multiplicity of
$\mathcal{F}$   is equal to the index of $\mathcal{F}$ (as defined
in \cite{cls}) along a generic separatrix. Then the topological
invariance of the algebraic multiplicity of a dicritical
singularity is a consequence of the topological invariance of the
index along a curve, which is proved in \cite{cls}.    Recently has been considered the problem for some classes of equivalences. In \cite{rosas2} is proved the invariance of the algebraic multiplicity under equivalences that are differentiable at the singular point and in \cite{rosas1} is solved the problem for 1-foliations in any dimension but only for  $C^1$ equivalences.

The following theorem is the first result of this work.
   \begin{thm}\label{main result}
   Let $\mathcal{F}$ and $\widetilde{\mathcal{F}}$ be holomorphic foliations by curves on neighborhoods $\mathfrak{U}$ and $\widetilde{\mathfrak{U}}$  of   $0\in\mathbb{C}^2$, respectively. Suppose that  $\mathcal{F}$ and $\widetilde{\mathcal{F}}$ have isolated singularities at $0\in\mathbb{C}^2$. Let $\mathfrak{h}:\mathfrak{U}\rightarrow\widetilde{\mathfrak{U}}$ be a topological equivalence between $\mathcal{F}$ and $\widetilde{\mathcal{F}}$. Assume  that $\mathfrak{h}$ is \textbf{bilipzchitz}, that is, there are positive constants $m,M$ such that $$m|z_1-z_2|\le|\mathfrak{h}(z_1)-\mathfrak{h}(z_2)|\le M|z_1-z_2|$$ for all $z_1,z_2\in {\mathfrak{U}}$. Then the algebraic multiplicity of $\mathcal{F}$ and $\widetilde{\mathcal{F}}$ at $0\in\mathbb{C}^2$ are the same.
   \end{thm}
It is worth mentioning that the  bilipschitz hypothesis is also used in \cite{risler trotman} to give a positive answer for the  Zariski's Multiplicity Problem.

Another object associated to a non-1-dicritical foliation\footnote{That is, after a single blow up the exceptional divisor is invariant by the foliation. } is its projective holonomy representation.  Cerveau
and Sad in \cite{CS} pose the following problem: Assuming  $\mathcal{F}$
is a non-1-dicritical generalized curve, it is true that the projective holonomy groups
of $\mathcal{F}$ and $\mathcal{F}'$ are topologically conjugated?
Also in \cite{CS} the authors give a positive answer for a generic
class of foliations $\mathcal{F}$ and assuming that $h$ is a topologically
trivial deformation. Stronger results in relation to this subject are obtained in  \cite{marin}, \cite{ORV}, \cite{MM2}, and \cite{rosas3}. We must remark the work of Mar\'{\i}n and Mattei (\cite{MM2}), who prove the topological invariance of the projective holonomy for a generic class of generalized curves, although the problem is  still unsettled  if we allow saddle node singularities after resolution. Thus, as a second result of this work we prove that the projective holonomy representation is a bilipschitz invariant for a large class of singularities (allowing saddle nodes after resolution). We say that a germ of holomorphic foliation at $\mathbb{C}^2$ belongs to the class $\mathfrak{G}$ if after a blow up at $0\in\mathbb{C}^2$ the exceptional divisor $E$ is invariant by the strict transform of the foliation and each  singularity at $E$ has some separatrix other than $E$.
\begin{thm}\label{second result}
Let $\mathcal{F}\in\mathfrak{G}$ and let  $\widetilde{\mathcal{F}}$ be a foliation topologically equivalent to $\mathcal{F}$ by a bilipschitz homeomorphism. Then $\widetilde{\mathcal{F}}\in\mathfrak{G}$ and the projective holonomy representations of $\mathcal{F}$ and $\widetilde{\mathcal{F}}$ are topologically conjugated.
\begin{rem}\label{aaa}
 Denote also by  $\mathcal{F}$ and $\widetilde{\mathcal{F}}$ the strict transforms of $\mathcal{F}$ and $\widetilde{\mathcal{F}}$ after a single blow up and let $E$   and $\widetilde{E}$ be the corresponding  exceptional divisors.  Let $E_*=E\backslash \mbox{Sing}(\mathcal{F})$ and $\widetilde{E}_*=\widetilde{E}\backslash \mbox{Sing}(\widetilde{\mathcal{F}})$. The conjugacy of the holonomy representations requires an isomorphism between the fundamental groups of $E_*$ and $\widetilde{E}_*$. In Theorem \ref{second result} this isomorphism is induced by a homeomorphism $\phi$ between  $E_*$ and $\widetilde{E}_*$ as follows. Let $\mathfrak{h}$ be the bilipschitz equivalence between $\mathcal{F}$ and $\widetilde{\mathcal{F}}$. Let $p_1,\ldots,p_k$ be the singularities of $\mathcal{F}$ in $E$. Let $S_1,\ldots,S_k$ be irreducible separatrices through $p_1,\ldots,p_k$ respectively all different from $E$. Let $\widetilde{S}_j=h(S_j)$ and let $\tilde{p}_j$ be the point where $\widetilde{S}_j$ meets the exceptional divisor $\widetilde{E}$. Clearly $\mathfrak{h}$ is a topological conjugacy between the singular curves $S=S_1\cup\ldots\cup S_k$ and $\widetilde{S}=\widetilde{S}_1\cup\ldots\cup\widetilde{S}_k$. Then by the main result of \cite{MM5} we have that $\mathfrak{h}$ is homotopic to another conjugacy  $\mathfrak{h}'$ which extends homeomorphically to $E$. On the other hand, in \ref{pt2} is proved that $\widetilde{\mathcal{F}}$ does not have singularities in $\widetilde{E}$ other than the $\tilde{p}_j$.  Therefore  $\phi:=\mathfrak{h}'|_{E_*}$ gives the desired homeomorphism.\end{rem}
\end{thm}
 \section{Itinerary and some remarks}\label{itinerary}

   Given foliations $\mathcal{F}$ and $\widetilde{\mathcal{F}}$ with isolated singularities at $0\in\mathbb{C}^2$, we say that $\mathcal{F}$ and $\widetilde{\mathcal{F}}$ are \emph{topologically equivalent} (at $0\in\mathbb{C}^2$) if there is an orientation preserving homeomorphism $\mathfrak{h}:{\mathfrak{U}}\rightarrow\tilde{{\mathfrak{U}}}$, $\mathfrak{h}(0)=0$  between neighborhoods of $0\in\mathbb{C}^2$,  taking leaves of  $\mathcal{F}$ to leaves of $\widetilde{\mathcal{F}}$. Such a homeomorphism is a \emph{topological equivalence} between $\mathcal{F}$ and $\widetilde{\mathcal{F}}$.
   Let $\pi:\widehat{\mathbb{C}^2}\rightarrow \mathbb{C}^2$ be the quadratic blow up at $0\in\mathbb{C}^2$, let $E=\pi^{-1}(0)$ be the exceptional divisor and let $\mathcal{F}_*$ and $\widetilde{\mathcal{F}}_*$  be the strict transforms of $\mathcal{F}$ and $\widetilde{\mathcal{F}}$ by $\pi$. We will always assume that $\mathcal{F}$ (and therefore $\widetilde{\mathcal{F}}$) is non-1-dicritical, that is, the exceptional divisor $E$ is invariant by  $\mathcal{F}_*$. We know that $\mathfrak{h}$ lifts to a homeomorphism \[
{h}={\pi}^{-1}\mathfrak{h}\pi:{\pi}^{-1}({\mathfrak{U}})\backslash E\rightarrow {\pi}^{-1}(\widetilde{{\mathfrak{U}}})\backslash E\]
 which takes leaves of ${\mathcal{F}}_*$ to leaves of $\widetilde{\mathcal{F}}_*$
and such that ${h}(w)\rightarrow E$ as $w\rightarrow E$.
Conversely, if $W$ and $\widetilde{W}$ are neighborhoods of $E$
 and $\bar{h}:W\backslash E\mapsto \widetilde{W}\backslash E$ is
a homeomorphism taking leaves of ${\mathcal{F}}_*$ to leaves
of $\widetilde{\mathcal{F}}_*$ and such that $\bar{h}(w)\rightarrow E$
as $w\rightarrow E$, then $\bar{h}$ induces a topological equivalence
between $\mathcal{F}$ and $\widetilde{\mathcal{F}}$. Thus, by simplicity, we
will say that any such $\bar{h}$ is a topological equivalence between
$\mathcal{F}$ and $\widetilde{\mathcal{F}}$. Moreover, when no confusion arises
we will often denote ${\mathcal{F}}_*$ and $\widetilde{\mathcal{F}}_*$
simply by $\mathcal{F}$ and $\widetilde{\mathcal{F}}$ respectively.

The starting point in the proof of Theorem \ref{main result} is the following result proved in \cite{rosas2}.
We say that a complex line $\mathfrak{P}$ passing through $0\in\mathbb{C}^2$ is $\mathcal{F}$-generic if the strict transform of $\mathfrak{P}$ by $\pi$ meets no singular point of $\mathcal{F}$ in the exceptional divisor.
\begin{thm}\label{preserva recta}
Let $\mathfrak{P}$ be a $\mathcal{F}$-generic complex line passing through $0\in\mathbb{C}^2$ and suppose that $\mathfrak{h}$ maps $\mathfrak{P}\cap {\mathfrak{U}}$ into a $\widetilde{\mathcal{F}}$-generic complex line. Then the algebraic multiplicity of $\mathcal{F}$ and $\widetilde{\mathcal{F}}$ are equal.
\end{thm}
The strict  transform  $P$ of $\mathfrak{P}$ is a Hopf fiber passing through a point $p\in E$. The hypothesis in Theorem \ref{preserva recta} basically means that, for some neighborhood $D$ of $p$ in $P$, we have that $h$ maps $D^* =D\backslash\{p\}$  into a Hopf  fiber. In particular this means that $h|_{D^*}$ extends continuously to $p$, which is the essence of the Hypothesis in Theorem \ref{preserva recta}. In general there is no extension of  $h|_{D^*}$ to $p$, even with the bilipschitz hypothesis. However, the strategy in the proof of Theorem \ref{main result} is to use $\mathfrak{h}$ and the bilipschitz hypothesis to construct another topological equivalence satisfying the hypothesis in Theorem \ref{preserva recta}.

Assume from now on that $\mathfrak{h}$ is bilipschitz.
By the Separatrix Theorem there exists an irreducible $\mathcal{F}$-invariant curve $S$ passing through $0\in\mathbb{C}^2$.  Then $\widetilde{S}=\mathfrak{h}(S)$ is an irreducible   $\widetilde{\mathcal{F}}$-invariant curve.  We can assume that $S$ and $\widetilde{S}$ are tangent to $\{(x,y)\in\mathbb{C}^2: y=0\}$ at $0\in\mathbb{C}^2$.
  Let $(t,x)$ be coordinates in $\widehat{\mathbb{C}^2}$ such that  $\pi(t,x)=(x,tx)\in\mathbb{C}^2$. In this chart the exceptional divisor $E$ is represented by $\{x=0\}$. Let  $\mathcal{S}$ and $\widetilde{\mathcal{S}}$ be the strict transform of $S$ and $\widetilde{S}$ by $\pi$. Then  $\mathcal{S}$ and $\widetilde{\mathcal{S}}$ pass through the point $(0,0)\in E$, so this point is singular for  $\mathcal{F}$ and $\widetilde{\mathcal{F}}$.
  The first step in order to prove Theorems \ref{main result} and \ref{second result} is the following proposition, proved in Section \ref{first step}.
  \begin{prop}
\label{control anular}Denote  $h(t,x)=(\tilde{t},\tilde{x})$. There exist positive constants $\rho$, $L_1<1<L_2$
and a continuous positive function $\delta:(0,\rho]\rightarrow\mathbb{R}$
such that, if $|t|=r\in(0,\rho]$ and $0<|x|\le \delta(r)$, then \[
L_{1}|t|\le |\tilde{t}|\le L_{2}|t|.\]

\end{prop}

\begin{rem}\label{remark1}Clearly we can assume $L_1=\frac{1}{L}$, $L_2=L$ for some constant $L>1$.
\end{rem}

Fix $r_{1},r_{2}\in(0,\rho]$ with $r_{1}<r_{2}<L r_2\le 1/2$ and such that:
\begin{enumerate}
\item $\mbox{Sing}(\mathcal{F})\cap\{|t|\le r_2\}=\{(0,0)\}$
\item $\mbox{Sing}(\widetilde{\mathcal{F}})\cap\{|t|\le L r_2\}=\{(0,0)\}$.
\end{enumerate}
From Proposition \ref{control anular} there exist constants $\delta_{0},\tilde{\delta}_{0}>0$
such that for any $(t,x)$ in the set \[
V:=\{r_{1}\leq|t|\leq r_{2},0<|x|\leq\delta_{0}\}\]
the point $(\tilde{t},\tilde{x})=h(t,x)$ satisfies
\begin{equation}
\frac{1}{L}|t|\le |\tilde{t}|\le L|t|\mbox{ and }|\tilde{x}|<\tilde{\delta}_{0}.\label{control en t}\end{equation}
In particular $V$ is mapped by $h$ into the set \[
\widetilde{V}:=\{\frac{1}{L}r_{1}\leq|t|\le Lr_{2},0<|x|\leq\tilde{\delta}_{0}\}.\]

\begin{rem}
\label{r2/r1 grande}Observe that we can take $r_1$, $r_2$ and $r_{1}/r_{2}$ arbitrarily
small. Moreover, fixed $r_{1}$ and $r_{2}$ we can take $\delta_{0}$
and $\tilde{\delta}_{0}$ arbitrarily small. These facts will be used
later.
\end{rem}

 The set $V$ is the local where we will modify $h$ in order to satisfy the hypothesis in Theorem \ref{preserva recta}. Suppose that the punctured Hopf fiber $D^*$ is taken in $V$. In general the set $h(D^*)$ can accumulates to a large set in the exceptional divisor, for example a point $w\in h(D^*)$ could oscillate infinitely many times around $\{t=0\}$ as  $w\rightarrow E$. Observe that, is $\tilde{\delta}_0$ is small enough, the dynamic of
  $\mathcal{F}|_{\overline{V}}$ is basically the suspension of the holonomy map of the leaf $E\cap\overline{V}$. We can see this dynamic as the 1-foliation induced by $\mathcal{F}$ on a set $$T=\{(t,x):|t|=r_{12},|x|\le \delta_1\}$$ with $r_1<r_{12}<r_2$ and $0<\delta_1\le\delta_0$. The main step to prove Theorem \ref{main result} will be locate the problem in the set $T$. Thus we will construct (Proposition \ref{regularizado}) another topological equivalence $\bar{h}$ which (for suitable $r_{12}$ and $\delta_1$)  maps $T$ into the set $$\widetilde{T}=\{(t,x)\in \widetilde{V}:|t|=\sqrt{{r_1}{r_2}}\}.$$ Clearly the 1-foliation induced by $\widetilde{\mathcal{F}}$ on $\widetilde{T}$ can be defined by a flow. The local punctured Hopf fibers in $T$ are mapped by $\bar{h}$ into sets which in general are not Hopf fibers in $\widetilde{T}$, but the idea is to  arrive to this situation by using the flow to move  each point of $\widetilde{T}$ the correct amount. This idea is realized in Section \ref{main proof} and we construct another topological equivalence $\bar{\bar{h}}$ mapping Hopf fibers in $T$ into Hopf fibers in $\widetilde{T}$, which is the final step in order to prove Theorem \ref{main result}. All the constructions summarized above can be performed in a neighborhood of each singularity  of $\mathcal{F}$ in $E$ having a separatrix other than $E$. This is the starting point in the proof of Theorem \ref{second result} in Section \ref{projective holonomy representation}.

\section{Proof of Proposition \ref{control anular}}\label{first step}

We use the following Lemma.

\begin{lem}\label{lem5}If $z_1,z_2\in \mbox{dom}(\mathfrak{h})$ and $z_2\neq 0$, then $$\frac{m}{M} \frac{|z_1-z_2|}{|z_2|}\le\frac{|\mathfrak{h}(z_1)-\mathfrak{h}(z_2)|}{|\mathfrak{h}(z_2)|}\le \frac{M}{m} \frac{|z_1-z_2|}{|z_2|}.$$
\end{lem}
\proof
We have $$|\mathfrak{h}(z_2)|=|\mathfrak{h}(z_2)-\mathfrak{h}(0)|\ge m|z_2-0|=m|z_2|,$$
hence $$|z_2|\le (1/m)|\mathfrak{h}(z_2)|.$$
Then $$|\mathfrak{h}(z_1)-\mathfrak{h}(z_2)|\le M|z_1-z_2|=M\frac{|z_1-z_2|}{|z_2|}|z_2|\le (M/m)\frac{|z_1-z_2|}{|z_2|}| \mathfrak{h}(z_2)|$$ and we obtain the right hand inequality. The other inequality  is similar.

 \subsection{Proof of Proposition \ref{control anular}}
 Since $S$ and $\widetilde{S}$ are tangent to $\{y=0\}$ at $0\in\mathbb{C}^2$, (using Weierstrass preparation theorem)  there exist neighborhoods $U$ and $\widetilde{U}$ of
 $0\in{\mathbb{C}}^2$ and a constant $l>0$ such that:

 \begin{itemize}
 \item[(i)] If $(x,y)\in S\cap U$, then $|y|\le l|x|^2$.
 \item[(ii)] If $(x,y)\in \widetilde{S}\cap\widetilde{U}$, then $|y|\le l|x|^2$.
 \item[(iii)] Given $(x,y)\in U$, there exists $y'\in{\mathbb{C}}$ such that $(x,y')\in S\cap U$.
 \item[(iv)] $\mathfrak{h}(U)\subset\widetilde{U}$.
 \end{itemize}  Let $(t,x)$ be such that $x\neq 0$  and  $(x,y):=\pi(t,x)=(x,tx)$ is contained in $U$. By (iii)  there exists $y'$ such that
 $(x,y')\in S\cap U$. Then
 \begin{eqnarray}
 |(x,y)-(x,y')|&=&|y-y'|\leq |y|+|y'|\le  |t||x|+l|x|^2\nonumber\\
 &=&(|t|+l|x|)|x|\leq (|t|+l|x|)|(x,y')|,\nonumber
 \end{eqnarray}
 hence $$\frac{|(x,y)-(x,y')|}{|(x,y')|}\leq |t|+l|x|.$$
 and by Lemma \ref{lem5}:$$\frac{|\mathfrak{h}(x,y)-\mathfrak{h}(x,y')|}{|\mathfrak{h}(x,y')|}\leq\frac{M}{m}(|t|+l|x|).$$
 Since $\mathfrak{h}(S)=\widetilde{S}$ and $(x,y')\in S\cap U$, by (iv) we have  $(\bar{x},\bar{y}):=\mathfrak{h}(x,y')\in\widetilde{S}\cap \widetilde{U}$, hence $|\bar{y}|\le l|\bar{x}|^2$ and therefore
\begin{eqnarray}|\mathfrak{h}(x,y)-\mathfrak{h}(x,y')|&\leq& \frac{M}{m}(|t|+l|x|)|(\bar{x},\bar{y})|\leq\frac{M}{m}(|t|+l|x|)(|\bar{x}|+|\bar{y}|)\nonumber\\
&\leq&\frac{M}{m}(|t|+l|x|)(|\bar{x}|+l|\bar{x}|^2),\nonumber
\end{eqnarray}
that is:
\begin{equation}\label{eq1}|\mathfrak{h}(x,y)-\mathfrak{h}(x,y')|\leq\frac{M}{m}(|t|+l|x|)(1+l|\bar{x}|)|\bar{x}|.\end{equation}
If $\mathfrak{h}(x,y)=(\tilde{x},\tilde{y})$, then
$$|\mathfrak{h}(x,y)-\mathfrak{h}(x,y')|=|(\tilde{x},\tilde{y})-(\bar{x},\bar{y})|=|(\tilde{x}-\bar{x},\tilde{y}-\bar{y})|\geq |\tilde{x}-\bar{x}|\geq |\bar{x}|-|\tilde{x}|$$
and by equation \ref{eq1}
$$|\bar{x}|-|\tilde{x}|\leq\frac{M}{m}(|t|+l|x|)(1+l|\bar{x}|)|\bar{x}|,$$
and so
\begin{equation}\label{eq2}
|\bar{x}|\leq\frac{1}{1-\frac{M}{m}(|t|+l|x|)(1+l|\bar{x}|)}|\tilde{x}|.\end{equation}
On the other hand:
\begin{equation}\label{eq3}|\mathfrak{h}(x,y')-(\bar{x},0)|=|(\bar{x},\bar{y})-(\bar{x},0)|=|\bar{y}|\leq l |\bar{x}|^2,\end{equation}
then
\begin{eqnarray} |\tilde{y}|&\leq&|(\tilde{x}-\bar{x},\tilde{y})|
=|(\tilde{x},\tilde{y})-(\bar{x},0)|=|\mathfrak{h}(x,y)-(\bar{x},0)|\nonumber\\
&\leq & |\mathfrak{h}(x,y)-\mathfrak{h}(x,y')|+|\mathfrak{h}(x,y')-(\bar{x},0)|, \nonumber
\end{eqnarray}
by equations \ref{eq1}  and  \ref{eq3}:
\begin{eqnarray}
|\tilde{y}|
&\leq&\frac{M}{m}(|t|+l|x|)(1+l|\bar{x}|)|\bar{x}|+l |\bar{x}|^2=\left(\frac{M}{m}(|t|+l|x|)(1+l|\bar{x}|)|+l |\bar{x}|\right)|\bar{x}|\nonumber
\end{eqnarray}
and by equation \ref{eq2}:
\begin{equation}|\tilde{y}|\leq \left(\frac{\frac{M}{m}(|t|+l|x|)(1+l|\bar{x}|)+l |\bar{x}|}{1-\frac{M}{m}(|t|+l|x|)(1+l|\bar{x}|)}\right)|\tilde{x}|,\end{equation}
hence \begin{equation}|\tilde{t}|\leq \frac{\frac{M}{m}(|t|+l|x|)(1+l|\bar{x}|)|+l |\bar{x}|}{1-\frac{M}{m}(|t|+l|x|)(1+l|\bar{x}|)}.\label{eqtxy}\end{equation}

By Puiseux Theorem there is a local parameterization of $(S,0)$ of
the form $\psi(\zeta)=(\zeta^{n},\psi_{2}(\zeta))$ defined on a small
disc $|\zeta|\leq\varrho$. Denote $\mathfrak{h}=(\mathfrak{h}_{1},\mathfrak{h}_{2})$.
Given $s\in[0,\varrho^n]$ define \[
f(s)=\sup\{|\mathfrak{h}_{1}(\psi(\zeta))|:|\zeta|\leq s^{\frac{1}{n}}\}.\]
It is easy to see that $f$ is a strictly increasing continuous function.
Observe that $$|\bar{x}|=|\mathfrak{h}_{1}(x,y')|=|\mathfrak{h}_{1}(\psi(x^{\frac{1}{n}}))|\leq f(|x|).$$
Using this fact in equation \ref{eqtxy} we obtain \begin{equation}
|\tilde{t|}\le\frac{\frac{M}{m}(|t|+l|x|)(1+lf(|x|))+lf(|x|)}{1-\frac{M}{m}(|t|+l|x|)(1+lf(|x|))}.\label{control anular 1}\end{equation}
 Let $f^{-1}$ be the inverse function of $f$ and consider the strictly
increasing function $\eta(r)=\min\{f^{-1}(r),r/l\}$ defined on an interval $[0,\rho]$. Suppose now
$|t|=r\in(0,\rho]$ and $|x|\leq\eta(r)$. Using the inequalities $|x|\leq r/l$
and $|x|\le f^{-1}(r)$ in equation \ref{control anular 1} we obtain
\[
|\tilde{t|}\le\frac{\frac{M}{m}(r+l(r/l))(1+lf(f^{-1}(r)))+lf(f^{-1}(r))}{1-\frac{M}{m}(r+l(r/l))(1+lf(f^{-1}(r)))}=\frac{2\frac{M}{m}r(1+lr)+lr}{1-2\frac{M}{m}r(1+lr)}.\]
 Clearly we can take $\rho$ small enough such that $(1+lr)\le2$
and $2\frac{M}{m}r(1+lr)\le1/2$ and therefore we obtain\[
|\tilde{t|}\le\frac{2\frac{M}{m}r(2)+lr}{1-(\frac{1}{2})}=Lr,\]
 were $L=(8\frac{M}{m}+2l)$.

If we apply the above arguments to $\mathfrak{h}^{-1}$ we find positive
constants $\tilde{\rho}$, $\widetilde{L}$ and a strictly increasing
continuous function $\tilde{\eta}:[0,\tilde{\rho}]\rightarrow[0,+\infty]$
such that $|t|\leq\widetilde{L}|\tilde{t}|$, whenever $|\tilde{t}|=\tilde{r}\in[0,\tilde{\rho}]$
and $|\tilde{x}|\le\tilde{\eta}(\tilde{r})$. Define the functions
$f':[0,\rho]\rightarrow[0,+\infty]$ and $f_{r}:[0,\eta(r)]\rightarrow[0,+\infty]$
($r\in(0,\rho]$) by $f'(0)=0$, $f_{r}(0)=0$ and \[
f'(s)=\inf\{|\tilde{t}|:|t|=s,0<|x|\le\eta(s)\},\]
 \[
f_{r}(s)=\sup\{|\tilde{x}|:|t|=r,0<|x|\le s\}\]
 if $s>0$. Clearly $f'(s)\le Ls\le L\rho$ and by reducing $\rho$ if necessary we can assume $f'(s)\le \widetilde{\rho}$, hence $\tilde{\eta}\circ f'$ is well defined.  For $r\in(0,\rho]$ define $\delta'(r)>0$ as follows.
If $f_{r}(\eta(r))\leq\tilde{\eta}\circ f'(r)$ define $\delta'(r)=\eta(r)$,
otherwise we choose $\delta'(r)$ such that $f_{r}(\delta'(r))=\tilde{\eta}\circ f'(r)$;
this is possible because the function $f_{r}$ is continuous and $f_{r}(0)=0$.
In any case we have $f_{r}(\delta'(r))\le\tilde{\eta}\circ f'(r)$. Moreover,
since $f_{r}$ is increasing we have $\delta'\le\eta$. Consider
now $(t,x)$ such that $|t|=r\in(0,\rho]$ and $|x|\le\delta'(r)$.
Then \[
|\tilde{x}|\leq\sup\{|\tilde{x}|:|t|=r,0<|x|\le\delta'(r)\}=f_{r}(\delta'(r))\le\tilde{\eta}\circ f'(r).\]But $\tilde{\eta}\circ f'(r)\le\tilde{\eta}(|\tilde{t}|)$ because
$\tilde{\eta}$ is increasing, hence
$|\tilde{x}|\le\tilde{\eta}(|\tilde{t}|)$
and by definition of $\tilde{\eta}$ we have $|t|\le\tilde{L}|\tilde{t}|$.
Since $\delta'\le\eta$ we have $|x|\le\eta (r)$, hence $|\tilde{t}|\le L|t|$. Therefore \[
L_{1}|t|\le|\tilde{t}|\le L_{2}|t|,\]  where $L_{1}=1/\tilde{L}$ and
$L_{2}=L$. Now it is sufficient to find a continuous function $\delta$
such that $0<\delta\leq\delta'$. By using a partition of unity it
is easy to see that it suffices to show that ${\displaystyle \liminf_{r'\rightarrow r}\delta'(r')>0}$
for all $r\in(0,\rho]$. Suppose by contradiction that there is a
sequence $(r_{n})$ with $r_{n}\rightarrow r$ and $\delta'(r_{n})\rightarrow0$.
By definition we have $\delta'(r_{n})=\eta(r_{n})$ or $f_{r_{n}}(\delta'(r_{n}))=\tilde{\eta}\circ f'(r_{n})$.
But $\delta'(r_{n})=\eta(r_{n})$ does not occur infinitely many times
because $\eta$ is continuous and $\eta(r)>0$. Then assume that \begin{equation}f_{r_{n}}(\delta'(r_{n}))=\tilde{\eta}\circ f'(r_{n})\label{pasto}\end{equation}
for all $n\in\mathbb{N}$ and suppose moreover that $f'(r_{n})\rightarrow\epsilon\ge0$.
As we will show at the end of the proof, $f_{r}(s)$ depends continuously
on $(r,s)$. Then, if $n\rightarrow\infty$ from equation \ref{pasto} we obtain $f_{r}(0)=\tilde{\eta}(\epsilon)$.
But $f_{r}(0)=0$ and therefore $\epsilon=0$, that is, $f'(r_{n})\rightarrow 0$.
By definition of $f'$ we may take $w_{n}=(t_{n},x_{n})$ with $|t_{n}|=r_{n}$,
$0<|x_{n}|\leq\eta(r_{n})$ and such that $|\widetilde{t_{n}}|\rightarrow 0$
as $n\rightarrow\infty$,  $(\tilde{t}_n\tilde{x}_n)=h(t_n,x_n)$. We can assume $t_{n}\rightarrow\bar{t}$,
$x_{n}\rightarrow\bar{x}$ with $|\bar{t}|=r$, $0\le|\bar{x}|\le\eta(r)$.
But $\bar{x}\neq0$ implies by continuity of $h$ that $|\widetilde{t_{n}}|\rightarrow|\tilde{\bar{t}}|>0$, where $(\tilde{\bar{t}},\tilde{\bar{x}})=h(\bar{t},\bar{x})$,
contradiction. Then we have $x_{n}\rightarrow0$. Take $\tilde{r}>0$
such that $\widetilde{L}\tilde{r}<r$ and set $$\widetilde{T}=\{(t,x):|t|=\tilde{r},0<|x|\leq\tilde{\eta}(\tilde{r})\}.$$
As we have seen $h^{-1}(\widetilde{T})$ is contained in $\{|t|\le\widetilde{L}\tilde{r}\}$.
Then, since $\widetilde{L}\tilde{r}<r$, provided $n$ is large enough
the point $w_{n}$ can be connected with the set $$T=\{(t,x):|t|=r,0<|x|\le\eta(r)\}$$
by a curve $\alpha_{n}$ contained in the set $$\{\widetilde{L}\tilde{r}<|t|\leq r,0<|x|\}.$$
In particular $\alpha_{n}$ does not meet the set $h^{-1}(\widetilde{T})$
and consequently the curve $h(\alpha_{n})$ does not meet $\widetilde{T}$.
Moreover, since $\alpha_{n}$ meets $T$ and $h(T)$ is contained
in $$\{L_{1}r\le|t|\}=\{\frac{1}{\widetilde{L}}r\le|t|\}\subset\{\tilde{r}<|t|\},$$
then $h(\alpha_{n})$ meets the set $\{\tilde{r}<|t|\}$. Observe
that $\alpha_{n}$ may be taken arbitrarily close to the exceptional
divisor for $n$ large enough. Then by connectedness  we see that, since
$h(\alpha_{n})$ does not meet $\widetilde{T}$, necessarily $h(\alpha_{n})$
is contained in the set $\{\tilde{r}<|t|\}$. In particular $h(w_{n})\in\{\tilde{r}<|t|\}$
and therefore $|\tilde{t}_n|>\tilde{r}>0$ for $n\in\mathbb{N}$
large enough, which is a contradiction. It remains to prove that $f_{r}(s)$
depends continuously on $(r,s)$. Fix $r_{0}>0$, $s_{0}\in(0,\eta(r_{0}))$.
Since $h(w)$ tends to the exceptional divisor as $w$ tends to the exceptional divisor we
may find $\varepsilon>0$ such that $$f_{r}(s)=\sup\{|\tilde{x}|:|t|=r,\varepsilon<|x|\le s\}$$
for all $(r,s)$ close enough to $(r_{0},s_{0})$. From here the continuity at $(r_{0},s_{0})$  is
easy to establish. If $s_{0}=0$
clearly we have $f_{r_{0}}(s_0)=0$. If $\varepsilon>0$ is small enough
the set $$V=\{r-\varepsilon\leq|t|\le r+\varepsilon,0<|x|\leq\varepsilon\}$$
is mapped by $h$ arbitrarily close to the exceptional divisor. Then
$|\tilde{x}|$ is uniformly small for $(t,x)\in V$ and the continuity
at $(r_{0},s_0$) follows.

\begin{cor}
\label{control anular 2} Let $L_1$ and $L_2$ be as in Proposition \ref{control angular}.Then there exist a constant $\tilde{\rho}>0$
and a continuous positive function $\tilde{\delta}:(0,\tilde{\rho}]\rightarrow\mathbb{R}$
such that, if $|\tilde{t}|=r\in(0,\tilde{\rho}]$ and $0<|\tilde{x}|\le \tilde{\delta}(r)$, then \[
L_{1}|\tilde{t}|\le |{t}|\le L_{2}|\tilde{t}|.\]

\end{cor}
\begin{proof}
This is a corollary of the proof of Proposition \ref{control angular}.
\end{proof}
\begin{rem} As in remark \ref{remark1}  we will assume $L_1=1/L,L_2=L$.
\end{rem}

\section{Complex time distortion}\label{distorcion}

In this section we fix complex flows associated to the foliations on $V$ and $\widetilde{V}$  and study the distortion of the complex time induced by $h$. Thus we establish Equation \ref{control en tau}  and state Proposition \ref{creciente-decreciente} which give us some  estimations for the complex time distortion.

Since the foliations $\mathcal{F}$ and $\widetilde{\mathcal{F}}$
are tangent to the exceptional divisor, we can assume that  $\mathcal{F}|_V$ and $\widetilde{\mathcal{F}}|_{\widetilde{V}}$ are generated by holomorphic vector fields $$Z=t\frac{\partial}{\partial t}+xQ\frac{\partial}{\partial x}$$ and $$\widetilde{Z}=t\frac{\partial}{\partial t}+x\widetilde{Q}\frac{\partial}{\partial x},$$ where $Q$   and $\widetilde{Q}$ are holomorphic functions on $\overline{V}$ and $\overline{\widetilde{V}}$ respectively\footnote{Then $Z$ and $\widetilde{Z}$ are defined on some neighborhoods of $\overline{V}$ and $\overline{\widetilde{V}}$, respectively.}.
Consider the complex flows $\phi^{T}=(\phi_{1}^{T},\phi_{2}^{T})$
and $\tilde{\phi}^{T}=(\tilde{\phi}_{1}^{T},\tilde{\phi}_{2}^{T})$ associated to $Z$ and $\widetilde{Z}$ respectively.
Clearly we have
$\phi_{1}^{T}(t,x)=\tilde{\phi}_{1}^{T}(t,x)=te^{T}$. Consider $w=(t,x)\in V$ and $$T\in\{\tau+{\theta i}\in\mathbb{C}:\ln \frac{r_{1}}{|t|}\le\tau\le\ln \frac{r_{2}}{|t|}\}$$
such that the path $\gamma_{w}^{T}(s)=\phi^{sT}(w)$, $s\in[0,1]$
is well defined and contained in $V$. Since $h\circ\gamma_{w}^{T}$ is a path starting
at $h(w)$ and contained in a leaf of $\widetilde{\mathcal{F}}|_{\widetilde{V}}$,
we can write $h\circ\gamma_{w}^{T}(s)=\tilde{\phi}^{\sigma(s)}(h(w))$,
where $\sigma:[0,1]\rightarrow\mathbb{C}$ is a continuous path with
$\sigma(0)=0$. Then define $$\widetilde{T}(w,T)=\sigma(1),$$
that is:``the
complex time between $h(w)$ and $h\phi^{T}(w)$''.

Put $\widetilde{T}=\tilde{\tau}+i\tilde{\theta}$.
Since $h\circ\gamma_{w}^{T}(1)=\tilde{\phi}^{\sigma(1)}(h(w))$, then
\[
h(\phi_{1}^{T}(w),\phi_{2}^{T}(w))=(\tilde{\phi}_{1}^{\widetilde{T}}h(w),\tilde{\phi}_{2}^{\widetilde{T}}h(w)).\]
and by  Equation \ref{control en t} for the point $\phi^{T}(w)$:
\[
\frac{1}{L}|\phi_{1}^{T}(w)|\le |\tilde{\phi}_{1}^{\widetilde{T}}h(w)|\le L|\phi_{1}^{T}(w)|.\]
Then, if $h(w)=(\tilde{t},\tilde{x})$ we have \[
\frac{1}{L}|t|e^{\tau}\le |\tilde{t}|e^{\tilde{\tau}}\le L|t|e^{\tau}\]
 and together with equation \ref{control en t} we obtain
 \begin{equation}
\tau-2\ln L\le \tilde{\tau}(w,T)\le \tau+2\ln L,\label{control en tau}
\end{equation}
which give us a control on the real part of the complex time $\widetilde{T}$
only in terms of the real part of the complex time $T$. Sections
\ref{sec:Consequences-of-the} and \ref{sec:Proof-of-Proposition}
are devoted to prove the following proposition, which establish a kind
of control on the imaginary part $\tilde{\theta}$ of $\widetilde{T}$.
\begin{prop}
\label{creciente-decreciente} There exist increasing homeomorphisms
$\vartheta_{1},\vartheta_{2}:\mathbb{R}\rightarrow\mathbb{R}$ such
that one of the following situations holds:
\begin{enumerate}
\item $\vartheta_{1}(\theta)\le\tilde{\theta}(w,T)\le\vartheta_{2}(\theta)$
for any $(w,T)$ such that $\tilde{\theta}$ is defined.
\item $\vartheta_{1}(\theta)\le-\tilde{\theta}(w,T)\le\vartheta_{2}(\theta)$
for any $(w,T)$ such that $\tilde{\theta}$ is defined.
\end{enumerate}
\end{prop}
\begin{rem}
Case 1 and case 2 in Proposition $\ref{creciente-decreciente}$ happen
according to $h$ preserves or reverses the natural orientations of
the leaves.
\end{rem}

\section{\label{sec:Consequences-of-the}Lipschitz condition along the leaves}

In this section we study the consequences of the Lipschitz condition along the leaves and
the following proposition is the main result of this section.  We use this result in the proof of Proposition \ref{creciente-decreciente} in Section \ref{sec:Proof-of-Proposition}. Given a rectifiable path $\gamma$ in $V$, we denote by $l(\gamma)$  the length of $\gamma$ induced by the norm $|(t,x)|=|t|+|x|$.
\begin{prop}
\label{lipschitz curva}There exists a constant $M_0>0$ with the following property. If $\gamma:[0,1]\rightarrow V$ is a continuous
rectifiable path contained in a leaf of $\mathcal{F}$, then%
 \[
\frac{1}{M_{0}}l(\gamma)\leq l(h\circ\gamma)\le M_{0}l(\gamma).\]
\end{prop}
We need some lemmas.

\begin{lem}\label{bu1} For $j=1,2$  put
$w_j=(t_j,x_j)$, $z_j=\pi (w_j)=(x_j,t_jx_j)$ and suppose $|x_j|,|t_j|\le\frac{1}{2}$.
\begin{enumerate}
\item If ${\mathfrak{x}}=\max\{|x_{1}|,|x_{2}|\}\neq 0$, then $$|w_{1}-w_{2}|\leq\frac{1}{{\mathfrak{x}}}|z_{1}-z_{2}|.$$
\item If $|x_{1}-x_{2}|\leq\eta|x_{1}||t_{1}-t_{2}|$ for some $\eta>0$,
then $$|z_{1}-z_{2}|\leq(2\eta+1)|x_{1}||w_{1}-w_{2}|.$$
\end{enumerate}
\end{lem}
\begin{proof}
we have \begin{eqnarray*}
|z_{1}-z_{2}| & = & |(x_{1},t_{1}x_{1})-(x_{2},t_{2}x_{2})|=|x_{1}-x_{2}|+|t_{1}x_{1}-t_{2}x_{2}|\\
 & = & |x_{1}-x_{2}|+|(t_{1}-t_{2})x_{1}-t_{2}(x_{2}-x_{1})|\\
 & \ge & |x_{1}-x_{2}|+|x_{1}||t_{1}-t_{2}|-|t_{2}||x_{2}-x_{1}|,\end{eqnarray*}
but $|t_{1}|,|x_{1}|\le\frac{1}{2}$, then \begin{eqnarray*}
|z_{1}-z_{2}| & \ge & \frac{1}{2}|x_{1}-x_{2}|+|x_{1}||t_{1}-t_{2}|\\ & \ge & |x_{1}||x_{1}-x_{2}|+|x_{1}||t_{1}-t_{2}|\\ & = & |x_{1}||w_{1}-w_{2}|\end{eqnarray*}
 and therefore  $|w_{1}-w_{2}|\leq\frac{1}{|x_1|}|z_{1}-z_{2}|.$  In the same way $|w_{1}-w_{2}|\leq\frac{1}{|x_2|}|z_{1}-z_{2}|$ and item (1) follows.

On the other hand we have \begin{eqnarray*}
|z_{1}-z_{2}| & = & |x_{1}-x_{2}|+|(t_{1}-t_{2})x_{1}-t_{2}(x_{2}-x_{1})|\\
 & \le & |x_{1}-x_{2}|+|x_{1}||t_{1}-t_{2}|+|t_{2}||x_{2}-x_{1}|\\
 & \le & |x_{1}-x_{2}|+|x_{1}||t_{1}-t_{2}|+|x_{2}-x_{1}|=2|x_{1}-x_{2}|+|x_{1}||t_{1}-t_{2}|\\
 & \le & 2\eta|x_{1}||t_{1}-t_{2}|+|x_{1}||t_{1}-t_{2}|\le(2\eta+1)|x_{1}||t_{1}-t_{2}|\\
 & \le & (2\eta+1)|x_{1}||w_{1}-w_{2}|.\end{eqnarray*}

\end{proof}
\begin{rem}
By remark \ref{r2/r1 grande} we may assume $\delta_0,\tilde{\delta}_0\le\frac{1}{2}$. Then Lemma \ref{bu1} holds for points $w_j$ in $V$ or $\widetilde{V}$.

\end{rem}
\begin{lem}\label{bu3}
\label{lipschitz eta} For $j=1,2$ let $w_j=(t_j,x_j)\in V$ and denote $h(w_j)=(\tilde{t}_j,\tilde{x}_j)$. Given $\eta>0$, there is a constant $M_{0}=M_{0}(\eta,m,M)$
such that, if $|x_{1}-x_{2}|\le\eta|x_1||t_{1}-t_{2}|$ and $|\tilde{x}_{1}-\tilde{x}_{2}|\le\eta|\tilde{x}_1||\tilde{t}_1-\tilde{t}_2|$,
then \[
1/M_{0}|w_{1}-w_{2}|\le|h(w_{1})-h(w_{2})|\le M_{0}|w_{1}-w_{2}|.\]
\end{lem}

\begin{proof}
Put $z_j=\pi(w_j)$.
Since \[
|\mathfrak{h}(z_{1})|=|\tilde{x}_{1}|+|\tilde{t}_{1}||\tilde{x}_{1}|\le2|\tilde{x}_{1}|,\]
then \[
|\tilde{x}_{1}|\ge\frac{1}{2}|\mathfrak{h}(z_{1})|\ge\frac{m}{2}|z_{1}|\ge\frac{m}{2}|x_{1}|,\]
hence ${\tilde{{\mathfrak{x}}}}:=\mbox{Max}\{|\tilde{x}_1|,|\tilde{x}_2|\}\ge\frac{m}{2}|x_{1}|$. In the same way ${\tilde{{\mathfrak{x}}}}\ge\frac{m}{2}|x_{2}|$
and therefore \[
{\tilde{{\mathfrak{x}}}}\ge\frac{m}{2}{{\mathfrak{x}}},\]  where ${{\mathfrak{x}}}=\mbox{Max}\{|x_{1}|,|x_{2}|\}$.
Then by Lemma \ref{bu1}  we have\begin{eqnarray*}
|h(w_{1})-h(w_{2})| & \le & \frac{1}{{\tilde{{\mathfrak{x}}}}}|\mathfrak{h}(z_{1})-\mathfrak{h}(z_{2})|\le\frac{2}{m{{\mathfrak{x}}}}|\mathfrak{h}(z_{1})-\mathfrak{h}(z_{2})|\\
 & \le & \frac{2M}{m{{\mathfrak{x}}}}|z_{1}-z_{2}|\le\frac{2M}{m{{\mathfrak{x}}}}(2\eta+1)|x_{1}||w_{1}-w_{2}|\\
 & \le & \frac{2M}{m}(2\eta+1)|w_{1}-w_{2}|,\end{eqnarray*}
so we obtain the right hand inequality with $M_{0}=\frac{2M}{m}(2\eta+1)$.
The other hand is similar.
\end{proof}

Let $Q$ and $\widetilde{Q}$ be as in Section \ref{distorcion}. Since $|Q|$ and $|\widetilde{Q}|$ are bounded it is easy to see the following fact:

\begin{fact}\label{hecho}
There exist $\eta>0$ with the following property:  if $\Gamma(s)=(t(s),x(s))$ is a differentiable path contained in a leaf of $\mathcal{F}|_{V}$ or $\widetilde{\mathcal{F}}|_{\widetilde{V}}$, then $$|x'(s)|\leq\frac{\eta}{2}|x(s)||t'(s)|.$$
\end{fact}

\subsection{Proof of Proposition \ref{lipschitz curva}}
Let $$0=s_{0}<\ldots<s_{n}=1$$ be a partition of $[0,1]$. Denote $\gamma(s_{j})=(t_{j},x_{j})$.
If the partition is fine enough we have that each segment $[t_{j-1,}t_{j}]$
is contained in $$\{t\in\mathbb{C},{\frac{1}{L}}r_{1}\leq|t|\le {L}r_{2}\}.$$
Fix $j$ and let $\Gamma=(\Gamma_1,\Gamma_2):[0,1]\rightarrow V$
be a path tangent to $\mathcal{F}$ with $\Gamma(0)=\gamma(s_{j-1})$,
$\Gamma(1)=\gamma(s_{j})$ and such that $\Gamma_1(s)=(1-s)t_{j-1}+st_{j}$.
Then by the fact above we have
$$|x_{j}-x_{j-1}|=|\Gamma_2(1)-\Gamma_2(0)|\leq\int_{0}^{1}|\Gamma_2'(s)|ds\le\int_{0}^{1}\frac{\eta}{2}|\Gamma_2(s)||\Gamma_1'(s)|ds.$$

We can assume the partition to be  small enough such that $|\Gamma_2(s)|\le 2|x_j|$. Then
$$|x_{j}-x_{j-1}|\le\int_{0}^{1}\frac{\eta}{2}|\Gamma_2(s)||\Gamma_1'(s)|ds\le \eta |x_j|\int_{0}^{1}|\Gamma_1'(s)|ds,$$ that is:
\begin{equation}
|x_{j}-x_{j-1}|\le \eta|x_j||t_{j}-t_{j-1}|.\label{pendiente1}\end{equation}
Denote $h(\gamma(s_{j}))=(\tilde{t}_{j},\tilde{x}_{j}).$ Provided the partition
is fine enough and by working with the path $h\circ\gamma$ we prove
as above that \begin{equation}
|\tilde{x}_{j}-	\tilde{x}_{j-1}|\le \eta |\tilde{x}_j||\tilde{t}_{j}-\tilde{t}_{j-1}|.\label{pendiente2}\end{equation}
From equations (\ref{pendiente1}) and (\ref{pendiente2}) and Lemma
\ref{lipschitz eta} we obtain \[
\frac{1}{M_{0}}|\gamma(t_{j})-\gamma(t_{j-1})|\le|h\circ\gamma(t_{j})-h\circ\gamma(t_{j-1})|\leq M_{0}|\gamma(t_{j})-\gamma(t_{j-1})|\]
 and the proposition follows.

\section{\label{sec:Proof-of-Proposition}Proof of Proposition $\ref{creciente-decreciente}$}
 As a direct consequence of Fact \ref{hecho}, if $(t(s),x(s))$ is a path in a leaf of $\mathcal{F}|_{V}$ or $\widetilde{\mathcal{F}}|_{\widetilde{V}}$, we have  that  $|x'(s)|\le {\eta}|t'(s)|$. We start with the following proposition.
\begin{prop}
\label{control angular}There exist constants $a_{1},a_{2},b_{1},b_{2}>0$
such that \[
a_{1}|\theta|-b_{1}\leq|\tilde{\theta}(w,T)|\leq a_{2}|\theta|+b_{2}.\]
\end{prop}
\begin{proof}
Fix $(w,T)$, $T=\tau+{\theta i}$ such that $\gamma(s)=\phi^{Ts}(w)$, $s\in[0,1]$
is contained in $V$. Put $\gamma(s)=(t(s),x(s))$ and observe that
$t(s)=r(s)e^{i\alpha(s)}$, where $r$ and $\alpha$ are differentiable
monotone functions. Then\begin{eqnarray*}
l(\gamma) & = & \int_{0}^{1}|t'(s)|ds+\int_{0}^{1}|x'(s)|ds\\
& \le & \int_{0}^{1}|t'(s)|ds+\int_{0}^{1}\eta|t'(s)|ds\\
 & = & (1+\eta)\int_{0}^{1}|r'(s)e^{i\alpha(s)}+r(s)i\alpha'(s)e^{i\alpha(s)}|ds\\
 & \leq & (1+\eta)\int_{0}^{1}|r'|ds+(1+\eta)r_{2}\int_{0}^{1}|\alpha'|ds\\
 & \leq & (1+\eta)(r_{2}-r_{1})+(1+\eta)r_{2}|\theta|,\end{eqnarray*}
 since $r$ and $\alpha$ are monotone. By Proposition \ref{lipschitz curva}
we have\begin{equation}
l(h\circ\gamma)\leq M_{0}l(\gamma)\leq M_{0}(1+\eta)(r_{2}-r_{1})+M_{0}(1+\eta)r_{2}|\theta|.\label{eqlt}\end{equation}
On the other hand, take a differentiable path $\tilde{\gamma}(s)=(\tilde{r}(s)e^{i\tilde{\alpha}(s)},\tilde{x}(s))$
homotopic with fixed endpoints  to $h\circ\gamma(s)$ in $\widetilde{V}$ and such that
$l(h\circ\gamma)\ge l(\tilde{\gamma})-1$. Clearly $\tilde{\theta}(w,T)=\tilde{\alpha}(1)-\tilde{\alpha}(0)$.
Then \begin{eqnarray*}
l(h\circ\gamma) & \ge & l(\tilde{\gamma})-1=\int_{0}^{1}|\tilde{r}'e^{i\tilde{\alpha}}+\tilde{r}i\tilde{\alpha}'e^{i\tilde{\alpha}}|ds+\int_{0}^{1}|\tilde{x}'|ds-1\\
 & \ge & \int_{0}^{1}|\tilde{r}'+\tilde{r}i\tilde{\alpha}'|ds-1\ge\int_{0}^{1}|\tilde{r}\tilde{\alpha}'|-1\ge {\frac{1}{L}}r_{1}\int_{0}^{1}|\tilde{\alpha}'|ds-1\\
 & \ge & {\frac{1}{L}}r_{1}|\tilde{\alpha}(1)-\tilde{\alpha}(0)|-1={\frac{1}{L}}r_{1}|\tilde{\theta}|-1\end{eqnarray*}
and together with equation \ref{eqlt} we obtain \[
|\tilde{\theta}|\le L\frac{M_{0}(1+\eta)r_{2}}{{}r_{1}}|\theta|+L\frac{M_{0}(1+\eta)(r_{2}-r_{1})+1}{{}r_{1}},\]
 which give us the right inequality of the proposition. We can perform
the same arguments with $h^{-1}$ to obtain the left inequality.
\end{proof}
\noindent\textit{Proof of Proposition \ref{creciente-decreciente}.}
Fix a constant $c>0$ such that $a_{1}c-b_{1}>0$. From proposition
\ref{control angular} we have \[
0<a_{1}\theta-b_{1}\le|\tilde{\theta}(w,T)|\le a_{2}\theta+b_{2}\]
 whenever $\theta\ge c$. By connectedness we have two cases:
\begin{enumerate}
\item $0<a_{1}\theta-b_{1}\le\tilde{\theta}(w,T)\le a_{2}\theta+b_{2}$
whenever $\theta\ge c$.
\item $0<a_{1}\theta-b_{1}\le-\tilde{\theta}(w,T)\le a_{2}\theta+b_{2}$
whenever $\theta\ge c$.
\end{enumerate}
In the same way we have the following possibilities:
\begin{enumerate}
\item[(1')] $a_{2}\theta-b_{2}\le\tilde{\theta}(w,T)\le a_{1}\theta+b_{1}<0$
whenever $\theta\le-c$.
\item[(2')] $a_{2}\theta-b_{2}\le-\tilde{\theta}(w,T)\le a_{1}\theta+b_{1}<0$
whenever $\theta\le-c$.
\end{enumerate}
Suppose items 1 and 2' holds. Then we may find $w_{0}\in V$ and $c_{1},c_{2}\ge {c}$
such that \begin{equation}
\tilde{\theta}(w_{0},-c_{1}i)=\tilde{\theta}(w_{0},c_{2}i).\label{cero}\end{equation}
Denote $w_{1}=\phi^{-c_{1}i}(w_{0})$. Then, from equation \ref{cero}
it is easy to see that $$\tilde{\theta}(w_{1},(c_{1}+c_{2})i)=0,$$
which contradicts Proposition \ref{control angular}. In the same
way we see that items 2 and 1' does not simultaneously hold. Suppose
that we have 1 and 1'. The other case is similar. Define $f_{1}(\theta)=a_{2}\theta-b_{2}$,
$f_{2}(\theta)=a_{1}\theta+b_{1}$ for $\theta\le-c$, $f_{1}(\theta)=a_{1}\theta-b_{1}$,
$f_{2}(\theta)=a_{2}\theta+b_{2}$ for $\theta\ge c$ and linearly
extend $f_{1}$ and $f_{2}$ to $[-c,c]$. Clearly $f_{1}\le\tilde{\theta}\le f_{2}$
on $|\theta|\ge c$. Since $\tilde{\theta}$ is bounded on $|\theta|\le c$,
we can put $\vartheta_{1}=f_{1}-C$, $\vartheta_{2}=f_{2}+C$ for
suitable $C>0$ to obtain $\vartheta_{1}\le\tilde{\theta}\le\vartheta_{2}$.\qed

\section{Homological Compatibility}\label{homological compatibility}

In this section we describe the map $h:V\rightarrow \widetilde{V}$ at homology level  and discard any homological obstruction to perform the constructions in the rest of the paper. Precisely, the results of this section are used in the proofs of Proposition \ref{regularizado} and Proposition \ref{redressing}.  Note that Proposition \ref{homologico basico} below is a special version of a kind of results previously obtained in  \cite{MM2} (Theorem 6.2.1) and \cite{rosas3} (section 5).

\begin{prop}
Let $(a,b)\in V$, $(\tilde{a},\tilde{b})\in\widetilde{V}$ and consider
the loops $\alpha,\beta:[0,1]\rightarrow V$, $\tilde{\alpha},\tilde{\beta}:[0,1]\rightarrow\widetilde{V}$
given by $\alpha(s)=(ae^{s2\pi i},b)$, $\beta=(a,be^{s2\pi i})$,
$\tilde{\alpha}(s)=(\tilde{a}e^{s2\pi i},\tilde{b})$ and $\tilde{\beta}(s)=(\tilde{a},\tilde{b}e^{s2\pi i})$.
Put $\xi=1$ or $-1$ according to $h$ preserves or reverses the
natural orientation of the leaves. Then, in the first homology group
of $\widetilde{V}$ we have
\[
[h(\alpha)]=\xi\tilde{\alpha}\mbox{ and }[h(\beta)]=\xi\tilde{\beta}.
\]
\label{homologico basico}\end{prop}

Let $w\in V$ and consider the map $\phi_{w}$ defined by $\phi_{w}(T)=\phi^{T}(w)$
and whose domain is the connected component of $0\in\mathbb{C}$ of
the set where $\phi^{T}(w)$ is defined. Clearly we have to cases:
\begin{enumerate}
\item The map $\phi_{w}$ is injective, or
\item There exists $k>0$ such that $\phi_{w}$ has period $2k\pi i$. In
this case we say that $w$ has period $2k\pi i$.
\end{enumerate}
In the same way we define the map $\widetilde{\phi}_{w}$ for any
$w\in\widetilde{V}$. Observe that $\widetilde{T}(w,T)$ is defined
if $T$ is in the domain of $\phi_{w}$.

\begin{lem}
\label{homologico}Let $\xi$ as in Proposition \ref{homologico basico}. Then ${w}$
has period $2k\pi i$ if and only if ${h(w)}$ has
period $\xi2k\pi i$. Moreover, in this case \[
\widetilde{T}(w,T+2k\pi i)=\widetilde{T}(w,T)+\xi2k\pi i\]
 for any $T$ in the domain of $\phi_{w}$.
\end{lem}
\begin{proof}
This is an easy application of Proposition \ref{homologico basico}.
\end{proof}

\subsection{Proof of Proposition \ref{homologico basico}. }
For some $m_{1},n_{1},m_{2},n_{2}\in\mathbb{Z}$ we have
\begin{equation}
h(\alpha)=m_{1}\tilde{\alpha}+n_{1}\tilde{\beta}\mbox{ and }h(\beta)=m_{2}\tilde{\alpha}+n_{2}\tilde{\beta}\mbox{ in }H_{1}(\widetilde{V}).\label{ecuacion 1}
\end{equation}
 Take neighborhoods $W$ and $\widetilde{W}$ of the exceptional divisor
$E=\pi^{-1}(0)$ with the following properties:
\begin{enumerate}
\item $W$ contains $V$
\item $W\cap \mathcal{S}$ is closed in $W$ and it is homeomorphic to a disc
\item $h(W\backslash E)=\widetilde{W}\backslash E$
\item $\pi(W)$ and $\pi(\widetilde{W})$ are homeomorphic to balls.
\end{enumerate}
Let $S=\pi(W\cap \mathcal{S}$) and $\widetilde{S}=\mathfrak{h}(S)$.
Since $\pi(W)$ is homeomorphic to $\mathbb{C}^{2}$ and $S$
is closed in $\pi(W)$ and homeomorphic to $\mathbb{C}$, we have
by Alexander's duality that $H_{1}(\pi(W)\backslash S)\simeq\mathbb{Z}$.
Then, since $W^{*}=W\backslash(E\cup \mathcal{S})$ is homeomorphic to $\pi(W)\backslash S$
we have $H_{1}(W^{*})\simeq\mathbb{Z}$. In the same way, if $\widetilde{W}^{*}=\widetilde{W}\backslash(E\cup\widetilde{\mathcal{S}})$
we have $H_{1}(\widetilde{W}^{*})\simeq\mathbb{Z}$. Let $D\subset W$
and $\widetilde{D}\subset\widetilde{W}$ be small complex discs transverse
to $\mathcal{S}$ and $\widetilde{\mathcal{S}}$ respectively. If $D$ and $\widetilde{D}$
are small enough we have that $\partial D$ and $\partial\widetilde{D}$
are generators of $H_{1}(W^{*})$ and $H_{1}(\widetilde{W}^{*})$
respectively. Since $\mathfrak{h}$ preserves orientation it follows
from the topological invariance of the intersection number %
\footnote{See \cite{dold} p. 200.%
} that
\begin{equation}
h(\partial D)=\xi\partial\widetilde{D}\mbox{ in }H_{1}(\widetilde{W}^{*}).\label{ecuacion intersection number}
\end{equation}
Without loss of generality we can assume that $\tilde{\alpha}$ and
$\tilde{\beta}$ are contained in a set $\widetilde{V}_{\tilde{\epsilon}}=\{(t,x)\in\widetilde{V}:|x|\le\tilde{\epsilon}\}$
such that $\widetilde{V}_{\tilde{\epsilon}}\subset\widetilde{W}^{*}$.
Then we can write equations \ref{ecuacion 1} in $H_{1}(\widetilde{V}_{\tilde{\epsilon}})$
and consequently we have
\begin{equation}
h(\alpha)=m_{1}\tilde{\alpha}+n_{1}\tilde{\beta}\mbox{ and }h(\beta)=m_{2}\tilde{\alpha}+n_{2}\tilde{\beta}\mbox{ in }H_{1}(\widetilde{W}^{*}).\label{ecuacion 1-1}
\end{equation}
Observe that $\mathfrak{h}$ is a topological equivalence between the curves $S$ and $\widetilde{S}$ and therefore the resolutions of these curves are isomorphic. Furthermore, by the main result of \cite{MM5} there exist a topological equivalence between $S$ and $\widetilde{S}$ which extends homeomorphically after resolution. Thus, we deduce  that there exist $p,q\in\mathbb{Z}$ such
that
\begin{align}
 & \alpha=p\partial D\mbox{ and }\beta=q\partial D\mbox{ in }H_{1}(W^{*}),\label{ecuacion equisingular}\\
 & \tilde{\alpha}=p\partial\widetilde{D}\mbox{ and }\tilde{\beta}=q\partial\widetilde{D}\mbox{ in }H_{1}(\widetilde{W}^{*}).\nonumber
\end{align}
 Then, since $h(W^{*})=\widetilde{W}^{*}$ from equations \ref{ecuacion intersection number},
\ref{ecuacion 1-1} and \ref{ecuacion equisingular} we obtain
\begin{align}
 & \xi p=m_{1}p+n_{1}q,\label{ecuacion aritmetica}\\
 & \xi q=m_{2}p+n_{2}q.\nonumber
\end{align}

Put
\[
{\mathfrak{V}}_{\epsilon}=\{(t,x):|t|\le r_{2},\:0<|x|\le\epsilon\}
\]
 for $\epsilon>0$ and
\[
\widetilde{{\mathfrak{V}}}=\{(t,x):|t|\le L_{2}r_{2},\:0<|x|\le\tilde{\delta}_{0}\}.
\]
\\

\noindent\emph{Assertion. If $\epsilon>0$ is taken small enough
$h({\mathfrak{V}}_{\epsilon})$ is contained in $\widetilde{{\mathfrak{V}}}$.}\\

Take $w\in {\mathfrak{V}}_{\epsilon}$. If $w\in V$ clearly we have $h(w)\in\widetilde{V}\subset\widetilde{{\mathfrak{V}}}$,
so we assume $w\notin V$ and in particular
\[
w\in {\mathfrak{V}}_{\epsilon}^{0}=\{(t,x)\in {\mathfrak{V}}_{\epsilon}:|t|<r_{2}\}.
\]
From Proposition \ref{control anular} we see that $w$ can be connected
by a path $\gamma$ in ${\mathfrak{V}}_{\epsilon}^{0}$ with a point $w'\in {\mathfrak{V}}_{\epsilon}^{0}$
such that $h(w')\in\widetilde{{\mathfrak{V}}}$. It follows from Corollary \ref{control anular 2}
that the set%
\footnote{Clearly we can assume $Lr_{2}\le\tilde{\rho}$.
}
\[
\Sigma=\{(t,x):|t|=Lr_{2},\,0<|x|\le\tilde{\delta}(Lr_{2})\}
\]
 is mapped by $h^{-1}$ outside ${\mathfrak{V}}_{\epsilon}^{0}$. In particular
$h(\gamma)$ does not meet $\Sigma$. Then, since $\epsilon$ small
implies $h(\gamma)$ close to the exceptional divisor, we deduce that
$h(\gamma)$ can not leave $\widetilde{{\mathfrak{V}}}$ and therefore $h(w)\in\widetilde{{\mathfrak{V}}}$
for $\epsilon$ small enough.\\

Without loss of generality we can assume $\alpha,\beta\subset {\mathfrak{V}}_{\epsilon}$
and since $\alpha=0$ in $H_{1}({\mathfrak{V}}_{\epsilon})$ the assertion above
implies
\begin{equation}
h(\alpha)=0\mbox{ in }H_{1}(\widetilde{{\mathfrak{V}}}).\label{ecuacion cero-1}
\end{equation}
Let $\partial \mathcal{S}$ and $\partial\widetilde{\mathcal{S}}$ be positively oriented
loops in $\mathcal{S}$ and $\widetilde{\mathcal{S}}$ such that $\pi(\partial \mathcal{S})$ and
$\pi(\partial\widetilde{\mathcal{S}})$ are generators of $H_{1}(S\backslash\{0\})$
and $H_{1}(\widetilde{S}\backslash\{0\})$ respectively.
Clearly we have
\begin{equation}
h(\partial \mathcal{S})=\xi\partial\widetilde{\mathcal{S}}\mbox{ in }H_{1}(\widetilde{{\mathfrak{V}}}).\label{ecuacion separatrices}
\end{equation}

Again, since the resolutions of $S$ and $\widetilde{S}$ are isomorphic there exist $k\in\mathbb{Z}$ such that
\begin{align*}
 & \partial \mathcal{S}=k\beta\mbox{ in }H_{1}({\mathfrak{V}}_{\epsilon}),\\
 & \partial\widetilde{\mathcal{S}}=k\tilde{\beta}\mbox{ in }H_{1}(\widetilde{{\mathfrak{V}}}).
\end{align*}

Then, since $h({\mathfrak{V}}_{\epsilon})\subset\widetilde{{\mathfrak{V}}}$ we have
\begin{align*}
h(\partial \mathcal{S}) & =kh(\beta)\mbox{ in }H_{1}(\widetilde{{\mathfrak{V}}}),\\
\partial\widetilde{\mathcal{S}} & =k\tilde{\beta}\mbox{ in }H_{1}(\widetilde{{\mathfrak{V}}})
\end{align*}
 and by \ref{ecuacion separatrices} we obtain
\begin{equation}
h(\beta)=\xi\tilde{\beta}\mbox{ in }H_{1}(\widetilde{{\mathfrak{V}}}).\label{betas}
\end{equation}
 Since $\widetilde{V}\subset\widetilde{{\mathfrak{V}}}$ we can write equations
\ref{ecuacion 1} in $H_{1}(\widetilde{{\mathfrak{V}}})$ and using that $\tilde{\alpha}=0$
in $H_{1}(\widetilde{{\mathfrak{V}}})$ we obtain:
\begin{equation}
h(\alpha)=n_{1}\tilde{\beta}\mbox{ and }h(\beta)=n_{2}\tilde{\beta}\mbox{ in }H_{1}(\widetilde{{\mathfrak{V}}}).\label{ecuacion 1-2}
\end{equation}
 Thus, from equations \ref{ecuacion cero-1}, \ref{betas} and \ref{ecuacion 1-2}
we obtain $n_{1}=0$ and $n_{2}=\xi$ and together with equation \ref{ecuacion aritmetica}
Proposition \ref{homologico basico} follows.

\section{\label{regularization}Regularization}
This section is devoted to prove Proposition \ref{regularizado}, which basically asserts that we can find a topological equivalence $\bar{h}$ mapping  $T$ into\footnote{ Note that $T$ and $\widetilde{T}$ denote now the solid tori defined in the end of Section \ref{itinerary} and are not the complex time.} $\widetilde{T}$, as we have announced in Section \ref{itinerary}. This proposition is the main construction of this work an is based on two key facts (Lemma \ref{lema empuja geodesica} and Proposition \ref{geodesica transversal}) which are proved in Section \ref{potencial theory}.

\begin{prop}
\label{regularizado} Close to the exceptional divisor there exists a topological equivalence $\bar{h}$  between $\mathcal{F}$ and $\widetilde{\mathcal{F}}$ with the following properties:
\begin{enumerate}
\item $\bar{h}=h$ outside $V$.
\item $\bar{h}$ maps  $V$ into $\widetilde{V}$.
\item \label{empuje limitado}On $V$ we have $\bar{h}(w)=\tilde{\phi}^{\overline{T}(w)}(h(w))$, where   $\overline{T}:V\mapsto\mathbb{C}$ is a continuous function such that $|\mbox{Im}(\overline{T})|$ is bounded by some constant $\tilde{\theta}_1>0$.
\item \label{r12}There exist $r_{12}\in(r_1,r_2)$ and $\delta_1\in(0,\delta_0)$ such that $\bar{h}$ maps the set $$T=\{(t,x)\in V:|t|=r_{12},
|x|\le\delta_1\}$$ into the set $$\widetilde{T}=\{(t,x)\in \widetilde{V}:|t|=\sqrt{{r_1}{r_2}}\}.$$
\end{enumerate}
\end{prop}

From now on we assume the first case in Proposition $\ref{creciente-decreciente}$,
that is: \begin{equation}
\vartheta_{1}(\theta)\le\tilde{\theta}(w,\tau+{\theta i})\le\vartheta_{2}(\theta)\mbox{ for all }(w,\tau+{\theta i})\mbox{ such that }\tilde{\theta}\mbox{ is defined}.\label{creciente}\end{equation}
For the second case the arguments works in the same way. Then, given
$\tilde{\theta}_{0}>0$ we can take $\theta_{0}>0$ and $\tilde{\theta}_{1}>0$
such that \begin{equation}
\tilde{\theta}(w,-\theta_{0})\leq-\tilde{\theta}_{0}\le\tilde{\theta}_{0}\le\tilde{\theta}(w,\theta_{0})\label{rectangulo grande}\end{equation}
and \begin{equation}
\tilde{\theta}(w,[-\theta_{0},\theta_{0}])\subset[-\tilde{\theta}_{1},\tilde{\theta}_{1}],\end{equation}
whenever $\tilde{\theta}(w,-\theta_{0})$ and $\tilde{\theta}(w,\theta_{0})$
are defined. We fix the constants $\tilde{\theta}_{0}$, $\tilde{\theta}_{1}$
and $\theta_{0}$ although we will specify later how $\tilde{\theta}_{0}$
is chosen.

Set \[
R=\{\tau+{\theta i}\in\mathbb{C}:0\le\tau\le\ln\frac{r_{2}}{r_{1}},|\theta|\le\theta_{0}\}\]
 and \[
\widetilde{R}=\{\tau+{\theta i}\in\mathbb{C}:-2\ln L\le\tau\le\ln\frac{r_{2}}{r_{1}}+2\ln L,|\theta|\le\tilde{\theta}_{1}\},\] where $L$ is as in Section \ref{itinerary} (see Proposition \ref{control anular}, Remark \ref{remark1} and Equation \ref{control en t}).
Define the set $$\partial V_{1}:=\{|t|=r_{1},0<|x|\le\delta_{0}'\}$$ with $\delta_0'>0$.\\

\noindent\emph{Assertion.  We can take  $\delta_0'$ small enough such that for all ${\mathfrak{w}}\in\partial V_{1}$ we have that $\widetilde{T}({\mathfrak{w}},T)$ is defined for all $T\in R$.}\\

 Let $\mathfrak{w}=(\mathfrak{t},\mathfrak{x})\in\partial V_1$. Recall that $\phi^T(\mathfrak{w})=(\phi_1^T(\mathfrak{w}),\phi_2^T(\mathfrak{w}))$ with $\phi_1^T(\mathfrak{w})=\mathfrak{t}e^{T}$ (see Section \ref{distorcion}).
 Fix $T\in R$. Take $s^+\in(0,1]$  maximal such that $\phi^{sT}(\mathfrak{w})$ is well defined and contained in $V$ for all $s\in[0,s^+)$.  It is easy to see that $$r_1\le |\phi_1^{sT}(\mathfrak{w})|\le r_2 \mbox{ for all } s\in[0,1].$$ Then, $\phi^{sT}(\mathfrak{w})$ is well defined and contained in $V$ for all $s\in[0,s^+)$  whenever $$f(s):=\phi_2^{sT}(\mathfrak{w})$$ is well defined and satisfies $|f(s)|\le\delta_0$ for all $s\in[0,s^+)$. It follows from the flow property that  $f$ satisfies the differential equation
 \[  f'= TQ(\mathfrak{t}e^{sT},f)f,\; f(0)=\mathfrak{x}, \]
 where the function $Q$ is as in Section \ref{distorcion}. Since $Q$ is bounded on $\overline{V}$ we find a constant $K>0$ such that $|f'(s)|\le K |f(s)|$ for all $s\in[0,s^+)$. From this we obtain the inequality  $|f(s)|\le |\mathfrak{x}|e^{Ks}$ for all $s\in[0,s^+)$ and therefore $|f(s)|\le |\mathfrak{x}|e^{K}$ for all $s\in[0,s^+)$. Then, by taking $\delta_0'$ small enough we can assume that $$|f(s)|\le \delta_0/2 \mbox{ for all } s\in[0,s^+).$$ Suppose that $s^+<1$. Then, from the elementary theory of ordinary differential equations there is a sequence $(s_n)\subset [0,s^+)$, $s_n\rightarrow s^+$ such that $|f(s_n)|\rightarrow\delta_0$, which is a contradiction. Therefore $\phi^{sT}(\mathfrak{w})$ is well defined and contained in $V$ for all $s\in[0,1)$. Recall that the vector field $Z$ defining $\phi$ is actually defined on a neighborhood  of $\overline{V}$ (see Section \ref{distorcion}). Then we deduce that     $\phi^{sT}(\mathfrak{w})$ is well defined and contained in $V$ for all $s\in[0,1]$ and the assertion follows.

 By the assertion above we have in particular that the segment $$I_{{\mathfrak{w}}}=\{\phi^{s\ln\frac{r_{2}}{r_{1}}}({\mathfrak{w}}),s\in[0,1]\}$$
is contained in $V$. Set $${\displaystyle V'=\bigcup_{{\mathfrak{w}}\in\partial V_{1}}I_{{\mathfrak{w}}}}.$$
We will redefine the function $h$ on each $I_{{\mathfrak{w}}}$. Naturally $\bar{h}(I_{{\mathfrak{w}}})$ must be a
segment with endpoints $\{h({\mathfrak{w}}),h(\phi^{\ln\frac{r_{2}}{r_{1}}}({\mathfrak{w}}))\}$.
Clearly $\widetilde{T}({\mathfrak{w}},T)\in\widetilde{R}$ for all $T\in R$ and
we can define the map $\widetilde{T}_{{\mathfrak{w}}}:R\rightarrow\widetilde{R}$,
$\widetilde{T}_{{\mathfrak{w}}}(T)=\widetilde{T}({\mathfrak{w}},T)$. It is not difficult to
see that we have the following properties:
\begin{enumerate}
\item $\widetilde{T}_{{\mathfrak{w}}}$ is a homeomorphism onto its image.
\item \label{contiene rectangulo}$\widetilde{T}_{{\mathfrak{w}}}(R)$ contains the rectangle
\begin{equation}\widetilde{R}_0=\{2\ln L\le\tau\le\ln\frac{r_{2}}{r_{1}}-2\ln L,-\tilde{\theta}_{0}\le\theta\le\tilde{\theta}_{0}\}.\label{rectangulo pequeno}\end{equation}
Here we use the equations \ref{rectangulo grande} and \ref{control en tau}.
\end{enumerate}
Clearly the points  $\widetilde{T}_{{\mathfrak{w}}}(0)=0$
and $\widetilde{T}_{{\mathfrak{w}}}(\ln\frac{r_{2}}{r_{1}}):=a_{{\mathfrak{w}}}$ are contained
in $\partial\widetilde{T}_{{\mathfrak{w}}}(R)$. Consider the Poincar\'e metric in
the interior of $\widetilde{T}_{{\mathfrak{w}}}(R)$ and let $\gamma$ be a geodesic
in $\widetilde{T}_{{\mathfrak{w}}}(R)$ with ${\displaystyle \lim_{s\rightarrow-\infty}\gamma(s)=0}$
and ${\displaystyle \lim_{s\rightarrow+\infty}\gamma(s)=a_{{\mathfrak{w}}}}$.
Although the parameterized geodesic $\gamma$ is not uniquely determined
the set $\Gamma_{{\mathfrak{w}}}=\gamma(\mathbb{R})\cup\{0,a_{{\mathfrak{w}}}\}$ is well defined
and depends continuously on ${\mathfrak{w}}$.
\begin{prop}
\label{geodesica transversal}If $r_{2}/r_{1}$ and $\tilde{\theta}_{0}$
are large enough, then $\Gamma_{{\mathfrak{w}}}$ intersects once and transversally
each line $\{\tau=c\}$ with $c$ in the interval $$[\frac{1}{2}\ln\frac{r_{2}}{r_{1}}-\ln L, \frac{1}{2}\ln\frac{r_{2}}{r_{1}}+\ln L]$$

\end{prop}
\noindent\emph{Idea of the proof.} For $r_{2}/r_{1}$ and $\tilde{\theta}_{0}$ arbitrarily large we have that, after "normalization", the region $\widetilde{T}_{{\mathfrak{w}}}(R)$ is arbitrarily close to $\widetilde{R}$ and the line $\{\tau=c\}$ is arbitrarily close to the center line $\{\tau=\frac{1}{2}\ln\frac{r_{2}}{r_{1}}\}$ of $\widetilde{R}$. Then the geodesic $\Gamma_{{\mathfrak{w}}}$ should be arbitrarily close to the geodesic $\{\theta=0\}$ in $\widetilde{R}$ in such way $\Gamma_{{\mathfrak{w}}}$ intersects once and transversely the line $\{\tau=c\}$.  The complete proof is given in Section \ref{potencial theory}.
\qed

\subsection{Proof of Proposition \ref{regularizado}}\label{aq}
Naturally we assume $r_{2}/r_{1}$ and $\tilde{\theta}_{0}$ large
enough according to Proposition \ref{geodesica transversal}. If $h(\mathfrak{w})=(h_1(\mathfrak{w}),h_2(\mathfrak{w}))$ it is easy to see that $$c_{\mathfrak{w}} =\ln\frac{\sqrt{r_1r_2}}{|h_1(\mathfrak{w})|}$$ belongs to the interval in Proposition \ref{geodesica transversal}. Then $\Gamma_{{\mathfrak{w}}}$ intersects once and transversely the line $\{\tau=c_{{\mathfrak{w}}}\}$.
Let
$o_{{\mathfrak{w}}}$ be the intersection point between $\Gamma_{{\mathfrak{w}}}$ and $\{\tau=c_{\mathfrak{w}}\}$.
Since this intersection is transversal the point $o_{{\mathfrak{w}}}$ depends
continuously on ${\mathfrak{w}}$. Now there is a unique parameterized geodesic $\gamma_{{\mathfrak{w}}}$
such that $${\displaystyle \lim_{s\rightarrow-\infty}\gamma_{{\mathfrak{w}}}(s)=0},\quad
{\displaystyle \lim_{s\rightarrow+\infty}\gamma_{{\mathfrak{w}}}(s)=a_{{\mathfrak{w}}}},\quad
\gamma_{{\mathfrak{w}}}(0)=o_{{\mathfrak{w}}}$$ and $\gamma_{{\mathfrak{w}}}$ depends continuously on ${\mathfrak{w}}$.
We are now in position to define $\bar{h}$ on $I_{{\mathfrak{w}}}$. Put $w=\phi^{s\ln\frac{r_{2}}{r_{1}}}({\mathfrak{w}})$, $s\in[0,1]$, fix an increasing diffeomorphism $f:(0,1)\rightarrow\mathbb{R}$ and
define
\begin{eqnarray*}
\bar{h}(w) & = & \tilde{\phi}^{\gamma_{{\mathfrak{w}}}(f(s))}(h({\mathfrak{w}})),\mbox{ if } s\in(0,1).\\
\bar{h}(w) & = & h(w),\mbox{ if } s=0\mbox{ or }-1.\end{eqnarray*}
Thus, $I_{\mathfrak{w}}$ is mapped onto $\tilde{\phi}^{\Gamma_{\mathfrak{w}}}(h({\mathfrak{w}}))$.
 Naturally  we set $\bar{h}=h$ outside ${\displaystyle V'=\bigcup_{{\mathfrak{w}}\in\partial V_{1}}I_{{\mathfrak{w}}}}$. If $$\overline{T}(w)=\gamma_{{\mathfrak{w}}}(f(s))-\widetilde{T}(w,s\ln\frac{r_{2}}{r_{1}})$$ by the flow property we have $$\bar{h}(w)=\tilde{\phi}^{\overline{T}(w)}(h(w)).$$
Then, since $\gamma_{{\mathfrak{w}}}(f(s))$  and  $\widetilde{T}(w,s\ln\frac{r_{2}}{r_{1}})$ are contained in  $\widetilde{R}$, we obtain item \ref{empuje limitado} of  Proposition \ref{regularizado}.
 Put $r_{12}=r_1(\frac{r_2}{r_1})^{s_{\mathfrak{w}}}$, where $s_\mathfrak{w}=f^{-1}(0)$. Suppose $w\in I_{\mathfrak{w}}$ is contained in the set $\{(t,x)\in V:|t|=r_{12}\}$. Then  $w={\phi}^{s_{\mathfrak{w}}\ln\frac{r_2}{r_1}}(\mathfrak{w})$, hence $$\bar{h}(w)=\tilde{\phi}^{\gamma_{\mathfrak{w}}(f(s_{\mathfrak{w}}))}(h(\mathfrak{w}))
 =\tilde{\phi}^{\gamma_{\mathfrak{w}}(0)}(h(\mathfrak{w}))=\tilde{\phi}^{o_{\mathfrak{w}}}(h(\mathfrak{w}))$$
 and $$|\tilde{\phi}_1^{o_{\mathfrak{w}}}(h(\mathfrak{w}))|
 =|h_1(\mathfrak{w})|e^{\mbox{Re}(o_{\mathfrak{w}})}=|h_1(\mathfrak{w})|e^{c_{\mathfrak{w}}}=\sqrt{r_1r_2},$$ which proves item \ref{r12}.

It is easy to see that, if  restricted to a small enough
neighborhood of the exceptional divisor, $\bar{h}$ has the following properties:
\begin{enumerate}
\item $\bar{h}$ is continuous.
\item $\bar{h}$ maps leaves of $\mathcal{F}$ into leaves of $\widetilde{\mathcal{F}}$.
\item $\bar{h}$ maps points in $V'$ into $h(V')$
\item $\bar{h}(w)$ tends to the exceptional divisor as $w$ tends
to the exceptional divisor.
\end{enumerate}
Then, in order to prove that $\bar{h}$ defines a topological equivalence
between $\mathcal{F}$ and $\widetilde{\mathcal{F}}$ it suffices
to show that $\bar{h}$ is injective, which reduces to show that the
sets $\{\bar{h}(I_{{\mathfrak{w}}})\}_{{\mathfrak{w}}\in\partial V_{1}}$ are disjoint. We will
use the following lemma, which is proved in Section \ref{potencial theory}.
\begin{lem}
\label{lema empuja geodesica}Let $U$ be a Jordan region in $\mathbb{C}$.
Consider the Poincar\'e metric in $U$ and let $\gamma$ be a geodesic
with ${\displaystyle \lim_{s\rightarrow-\infty}\gamma(s)=\zeta_{1}\in\partial U}$
and ${\displaystyle \lim_{s\rightarrow\infty}\gamma(s)=\zeta_{2}\in\partial U}$ ($\zeta_1\neq\zeta_2$).
Let $C\subset\partial U$ be one of the two segments determined by $\{\zeta_{1},\zeta_{2}\}$.
Let $\widetilde{U}$ be a Jordan region such that  $$U\subsetneq\widetilde{U}\mbox{ and
}C\subset\partial\widetilde{U}.$$
Consider the Poincar\'e metric on $	\widetilde{U}$ and let $\tilde{\gamma}$
be the geodesic in $\widetilde{U}$ with ${\displaystyle \lim_{s\rightarrow-\infty}\tilde{\gamma}(s)=\zeta_{1}}$  and  ${\displaystyle \lim_{s\rightarrow\infty}\widetilde{\gamma}(s)=\zeta_{2}}$.
Let $\Omega$ be the Jordan region bounded by $\gamma$ and $C$.
Then $\tilde{\gamma}$ is disjoint of $\overline{\Omega}$.
\end{lem}
Let ${\mathfrak{w}}_{1},{\mathfrak{w}}_{2}$ be different points in $\partial V_{1}$. In order
to prove that $\bar{h}(I_{{\mathfrak{w}}_{1}})\cap\bar{h}(I_{{\mathfrak{w}}{}_{2}})=\emptyset$,
since  $$\bar{h}(I_{{\mathfrak{w}}_{1}})\subset h(\phi^{R}({\mathfrak{w}}_{1})),\quad
\bar{h}(I_{{\mathfrak{w}}_{2}})\subset h(\phi^{R}({\mathfrak{w}}_{2})) \mbox{ and } h \mbox{
is injective,}$$ it suffices to consider the case $$\phi^{R}({\mathfrak{w}}_{1})\cap\phi^{R}({\mathfrak{w}}_{2})\neq\emptyset.$$
Reordering ${\mathfrak{w}}_{1},{\mathfrak{w}}_{2}$ if necessary we have the following:
\begin{enumerate}
\item ${\mathfrak{w}}_{2}=\phi^{{\alpha i}}({\mathfrak{w}}_{1})$ for some $\alpha>0$.\label{alpha}
\item $\phi^{T}({\mathfrak{w}}_{1})$ is defined for all $T$ in $$\mathcal{R}=\{\tau+\theta i:0\le\tau\le\ln\frac{r_{2}}{r_{1}},-\theta_{0}\le\theta\le\alpha+\theta_{0}\}.$$
\end{enumerate}
Let $$R_{\alpha}=\{\tau+{\theta i}:0\le\tau\le\ln\frac{r_{2}}{r_{1}},\alpha-\theta_{0}\le\theta\le\alpha+\theta_{0}\}$$
and consider the Jordan regions $\widetilde{T}_{{\mathfrak{w}}_{1}}(R)$, $\widetilde{T}_{{\mathfrak{w}}_{1}}(\mathcal{R})$
and $\widetilde{T}_{{\mathfrak{w}}_{1}}(R_{\alpha})$. By applying Lemma $\ref{lema empuja geodesica}$
to the regions $\widetilde{T}_{{\mathfrak{w}}_{1}}(R)\subsetneq\widetilde{T}_{{\mathfrak{w}}_{1}}(\mathcal{R})$
we see that the geodesic $\gamma_{1}$ in $\widetilde{T}_{{\mathfrak{w}}_{1}}(\mathcal{R})$
{}``joining'' the points $\widetilde{T}_{{\mathfrak{w}}_{1}}(0)$ and $\widetilde{T}_{{\mathfrak{w}}_{1}}(\ln\frac{r_{2}}{r_{1}})$
is disjoint of $\overline{\Omega_{0}}$, where $\Omega_{0}$ is the
region bounded by the curves:
\begin{enumerate}
\item The geodesic $\Gamma_{{\mathfrak{w}}_{1}}$,
\item $C_{1}^{0}:\widetilde{T}_{{\mathfrak{w}}_{1}}(-s\theta_{0}i)$, $s\in[0,1]$,
\item $C_{2}^{0}:\widetilde{T}_{{\mathfrak{w}}_{1}}(s\ln\frac{r_{2}}{r_{1}}-\theta_{0}i)$,
$s\in[0,1]$,
\item $C_{3}^{0}:\widetilde{T}_{{\mathfrak{w}}_{1}}(\ln\frac{r_{2}}{r_{1}}-s\theta_{0}i)$,
$s\in[0,1]$.
\end{enumerate}
Then, if $\Omega_{1}$ is the region bounded by $\gamma_{1}$, $C_{1}^{0}$,
$C_{2}^{0}$ and $C_{3}^{0}$ we have $\Gamma_{{\mathfrak{w}}_{1}}\subset\Omega_{1}$.
Since the geodesic $\gamma_{2}$ in $\widetilde{T}_{{\mathfrak{w}}_{1}}(\mathcal{R})$
"joining" the points $\widetilde{T}_{{\mathfrak{w}}_{1}}({\alpha i})$ and $\widetilde{T}_{{\mathfrak{w}}_{1}}(\ln\frac{r_{2}}{r_{1}}+{\alpha i})$
does not meet $\gamma_{1}$, we have that $\Omega_{1}$ is contained
in the region $\Omega_{2}$ bounded by the curves:
\begin{enumerate}
\item The geodesic $\gamma_{2}$,
\item $C_{1}^{1}:\widetilde{T}_{{\mathfrak{w}}_{1}}((1-s)\alpha i-s\theta_{0}i)$, $s\in[0,1]$,
\item $C_{2}^{1}:\widetilde{T}_{{\mathfrak{w}}_{1}}(s\ln\frac{r_{2}}{r_{1}}-\theta_{0}i)$,
$s\in[0,1]$,
\item $C_{3}^{1}:\widetilde{T}_{{\mathfrak{w}}_{1}}(\ln\frac{r_{2}}{r_{1}}+(1-s)\alpha i-s\theta_{0}i)$,
$s\in[0,1]$.
\end{enumerate}
By applying Lemma $\ref{lema empuja geodesica}$ to the regions $\widetilde{T}_{{\mathfrak{w}}_{1}}(R_{\alpha})\subsetneq\widetilde{T}_{{\mathfrak{w}}_{1}}(\mathcal{R})$
we have that the geodesic $\gamma_{3}$ in $\widetilde{T}_{{\mathfrak{w}}_{1}}(R_{\alpha})$
"joining" the points $\widetilde{T}_{{\mathfrak{w}}_{1}}({\alpha i})$ and $\widetilde{T}_{{\mathfrak{w}}_{1}}(\ln\frac{r_{2}}{r_{1}}+{\alpha i})$
is disjoint of $\overline{\Omega_{2}}$, hence $\gamma_{3}$ is disjoint of $\Gamma_{{\mathfrak{w}}_{1}}\subset\Omega_{1}\subset\Omega_{2}$.
By the definition of $\widetilde{T}_{w}$ and the flow property we
obtain the following relation: \[
\widetilde{T}_{{\mathfrak{w}}_{1}}({\alpha i}+T)=\widetilde{T}_{{\mathfrak{w}}_{1}}({\alpha i})+\widetilde{T}_{{\mathfrak{w}}_{2}}(T).\]

Then $\widetilde{T}_{{\mathfrak{w}}_{1}}(R_{\alpha})=\widetilde{T}_{{\mathfrak{w}}_{1}}({\alpha i})+\widetilde{T}_{{\mathfrak{w}}_{2}}(R)$
and, since the map $z\mapsto z+\widetilde{T}_{{\mathfrak{w}}_{1}}({\alpha i})$ is
a isometry between $\widetilde{T}_{{\mathfrak{w}}_{2}}(R)$ and $\widetilde{T}_{{\mathfrak{w}}_{1}}(R_{\alpha})$,
we deduce that $\gamma_{3}=\widetilde{T}_{{\mathfrak{w}}_{1}}({\alpha i})+\Gamma_{{\mathfrak{w}}_{2}}$.
Then\begin{eqnarray*}
\tilde{\phi}^{\gamma_{3}}(h({\mathfrak{w}}_{1})) & = & \tilde{\phi}^{\widetilde{T}_{{\mathfrak{w}}_{1}}({\alpha i})+\Gamma_{{\mathfrak{w}}_{2}}}(h({\mathfrak{w}}_{1}))\\
 & = & \tilde{\phi}^{\Gamma_{{\mathfrak{w}}_{2}}}(\tilde{\phi}^{\widetilde{T}_{{\mathfrak{w}}_{1}}({\alpha i})}(h({\mathfrak{w}}_{1}))\\
 & = & \tilde{\phi}^{\Gamma_{{\mathfrak{w}}_{2}}}(h({\mathfrak{w}}_{2}))\\
 & = & \bar{h}(I_{{\mathfrak{w}}_{2}}).\end{eqnarray*}
 At this point we consider two cases.
\begin{caseenv}
\item The map $\widetilde{\phi}_{h({\mathfrak{w}}_{1})}:T\mapsto\tilde{\phi}^{T}(h({\mathfrak{w}}_{1}))$
is injective. Then, since $\gamma_{3}$ is disjoint of $\Gamma_{{\mathfrak{w}}_{1}}$
we have that $\tilde{\phi}_{h({\mathfrak{w}}_{1})}(\gamma_{3})$ is disjoint of
$\tilde{\phi}_{h({\mathfrak{w}}_{1})}(\Gamma_{{\mathfrak{w}}_{1}})$, that is, $\bar{h}(I_{{\mathfrak{w}}_{2}})$
is disjoint of $\bar{h}(I_{{\mathfrak{w}}_{1}})$.
\item The map $\widetilde{\phi}_{h({\mathfrak{w}}_{1})}:T\mapsto\tilde{\phi}^{T}(h({\mathfrak{w}}_{1}))$
is not injective. In this case, provided $h({\mathfrak{w}}_{1})$ is close enough
to the exceptional divisor the leaf of $\widetilde{\mathcal{F}}|_{\widetilde{V}}$
through $h({\mathfrak{w}}_{1})$ is compact and by Lemma \ref{homologico} the
maps $T\mapsto\phi^{T}({\mathfrak{w}}_{1})$ and $T\mapsto\widetilde{\phi}^{T}(h({\mathfrak{w}}_{1}))$
have period $2\pi ki$ for some $k\in\mathbb{N}$ and $$\widetilde{T}_{{\mathfrak{w}}_{1}}(T+2k\pi i)=\widetilde{T}_{{\mathfrak{w}}_{1}}(T)+\xi 2k\pi i.$$
Therefore, for any $n\in\mathbb{Z}$, in the arguments above we can
change $\alpha$ by $\alpha+n2k\pi$, we put $\gamma_{3}^{k}$ instead $\gamma_3$
and thus we obtain that $\Gamma_{{\mathfrak{w}}_{1}}$ is disjoint of
\begin{eqnarray*}
\gamma_{3}^{k} & = & \widetilde{T}_{{\mathfrak{w}}_{1}}((\alpha+n\xi 2k\pi)i)+\Gamma_{{\mathfrak{w}}_{2}}\\
 & = & \widetilde{T}_{{\mathfrak{w}}_{1}}(\alpha i)+n\xi 2k\pi i+\Gamma_{{\mathfrak{w}}_{2}}\\
 & = & \gamma_{3}+n\xi 2k\pi i
 \end{eqnarray*}
for any $n\in\mathbb{Z}$. Then, since $\widetilde{\phi}^{\gamma_{3}}(h({\mathfrak{w}}_{1}))=\bar{h}(I_{{\mathfrak{w}}_{2}})$
we deduce that $\Gamma_{{\mathfrak{w}}_{1}}$ is disjoint of the inverse image
of $\bar{h}(I_{{\mathfrak{w}}_{2}})$ by $T\mapsto\widetilde{\phi}^{T}(h({\mathfrak{w}}_{1}))$.
Therefore $\bar{h}(I_{{\mathfrak{w}}_{1}})$
is disjoint of $\bar{h}(I_{{\mathfrak{w}}_{2}})$.
\end{caseenv}

\section{\label{main proof}Proof Theorem \ref{main result}}
Let $\bar{h}$ as in Section \ref{regularization}. As we have seen
in Proposition \ref{regularizado}, for some $\delta_{1}>0$ we have
that $\bar{h}$ maps the set $$T=\{|t|=r_{12},0<|x|\le \delta_{1}\}$$
into the set $$\widetilde{T}=\{|t|=\sqrt{r_{1}r_{2}},0<|x|\le \tilde{\delta}_{0}.\}$$
In this section we construct
a  topological equivalence $\bar{\bar{h}}$   between $\mathcal{F}$
and $\widetilde{\mathcal{F}}$ mapping the Hopf fibers in $T$ into
Hopf fibers in $\widetilde{T}$. Clearly this finishes the proof of
Theorem \ref{main result}. The key of the proof is the following
proposition  which shows that there is a Hopf fiber in $T$ whose image by $\bar{h}$
in $\widetilde{T}$ can be "redressing" to a Hopf fiber in $\widetilde{T}$.
\begin{prop}
\label{redressing} By reducing $\delta_1$  if necessary, there exist:
\begin{enumerate}
\item A punctured Hopf fiber $D^{*}=\{(t_{0},x):0<|x|<\delta_{1}\}$ with
$|t_{0}|=r_{12}$
\item A Hopf fiber $\widetilde{D}=\{(\tilde{t}_{0},x):|x|<\tilde{\delta}_{0}\}$
 with $|\tilde{t}_{0}|=\sqrt{r_{1}r_{2}}$
\item A real function $\sigma$ defined on $\mathcal{A}=\bar{h}(D^{*})$,
\end{enumerate}
such that the following properties hold:
\begin{enumerate}
\item $\tilde{\phi}^{s\sigma(\tilde{w})i}(\tilde{w})\in\widetilde{T}$ and $f(\tilde{w}):=\tilde{\phi}^{\sigma(\tilde{w})i}(\tilde{w})\in\widetilde{D}$
for all $\tilde{w}\in\mathcal{A}$ and all $s\in[0,1]$
\item $f:\mathcal{A}\rightarrow\widetilde{D}$ is a homeomorphism onto its
image
\item $f(\mathcal{A})=\Omega\backslash o$, where $o=(\tilde{t}_{0},0)\in\widetilde{D}$
and $\Omega$ is a topological disc
\item $f(\widetilde{w})$ tends to $o$ as $\tilde{w}$ tends to the exceptional
divisor.
\end{enumerate}
\end{prop}
\begin{proof}
We give the proof of this proposition in sections  \ref{redressing nh} and \ref{hyperbolic}.
\end{proof}

We perform the proof of Theorem \ref{main result} in two steps. \\

\noindent\textbf{Step 1.} As a first step, for some $\widetilde{T}_1\subset\widetilde{T}$  we construct a homeomorphism $f:\widetilde{T}_1\rightarrow\widetilde{T}$ preserving $\widetilde{\mathcal{F}}$ and redressing the image by ${\bar{h}}$ of the Hopf fibers in $T$.
Set $$C=\{|t|=r_{12},x=0\},$$ $$\widetilde{C}=\{|t|=\sqrt{r_1r_2},x=0\},$$ ${\zeta}_0=(t_0,0)$  and $\tilde{\zeta}_0=(\tilde{t}_0,0)$.
Consider the natural orientation on $C$.
Let $\gamma$
be a oriented circle in $T$ homotopic to $C$ in $\overline{T}$
and take a diffeomorphism $g:C\mapsto{\widetilde{C}}$, $g(\zeta_{0})=\tilde{\zeta}_0$
such that $g(C)$ is homotopic to $\bar{h}(\gamma)$ in $\overline{\widetilde{T}}$.
Put $$\widetilde{T}_1=\{(t,x)\in\widetilde{T}: 0<|x|\le \tilde{\delta}_1\}$$ and assume $\tilde{\delta}_1>0$
be such that
\begin{enumerate}
\item $\tilde{\phi}^{s2\pi i}(w)\in \widetilde{T}$ for all $w\in \widetilde{T}_1$, $s\in[-1,1]$
\item \label{234}$\phi^{s2\pi i}(w)\in T$ for all $w\in{{\bar{h}}}^{-1}(\widetilde{T}_1)$,
$s\in[-1,1]$.
\end{enumerate}
Given $\zeta\in C$, define $\vartheta(\zeta)\in[0,2\pi)$ by $\zeta=\phi^{\vartheta(\zeta)i}(\zeta_{0})$
and let ${\tilde{\vartheta}}(\zeta)\in\mathbb{R}$ be such that $\tilde{\phi}^{s{\tilde{\vartheta}}(\zeta)i}(\tilde {\zeta}_0)$,
$s\in[0,1]$ is a positive reparameterization of the path $g(\phi^{s\vartheta(\zeta)i}(\zeta_{0}))$,
$s\in[0,1]$. Clearly $\vartheta$ and ${\tilde{\vartheta}}$ are continuous
on $C\backslash\lbrace\zeta_0\rbrace$ and they have a simply discontinuity
at $\zeta_{0}$. Let $\pi_1$ be the projection $(t,x)\rightarrow t$
in $T$. Given $\tilde{w}\in \widetilde{T}_1$, put $\zeta(\tilde{w})=\pi_1\circ{\bar{h}}^{-1}(\tilde{w})$
and let $\theta(\tilde{w})\in\mathbb{R}$ be such that $\tilde{\phi}^{-s\theta(\tilde{w})i}(\tilde{w})$,
$s\in[0,1]$ is a positive reparameterization of $$\bar{h}(\phi^{-s\vartheta(\zeta(\tilde{w}))i}({\bar{h}}^{-1}(\tilde{w}))),\,
s\in[0,1].$$ From (\ref{234}) and the definition of $\theta$ it is easy
to see that $\tilde{\phi}^{-\theta(\tilde{w})i}(\tilde{w})\in\mathcal{A}$ for all $\tilde{w}\in \widetilde{T}_1$. Let $\sigma$ be the function given by Proposition \ref{redressing}.   We extend $\sigma$ to $\widetilde{T}_1$ by putting:
  \begin{equation}
\sigma(\tilde{w})=-\theta(\tilde{w})+\sigma (\tilde{\phi}^{-\theta(\tilde{w})i}(\tilde{w}))+{\tilde{\vartheta}}(\zeta(\tilde{w})).\label{tau}\end{equation}
\\

\noindent\textit{Assertion.
$\sigma$ is continuous and $\phi^{s\sigma(\tilde{w})i}(\tilde{w})\in \widetilde{T}$ for all $\tilde{w}\in {\widetilde{T}_1}$,
$s\in[0,1]$.} \\

Fix $\tilde{w}_1\in\mathcal{A}$. It is sufficient to show that $\sigma(\tilde{w})\rightarrow\sigma({\tilde{w}_1})$
whenever $\tilde{w}\rightarrow{\tilde{w}_1}\in\mathcal{A}$ with $\pi<\vartheta(\zeta(\tilde{w}))<2\pi$.
If $\vartheta(\zeta(\tilde{w}))\rightarrow2\pi$ we have that $\theta(\tilde{w})\rightarrow\theta_{0}$,
where $\theta_{0}$ is such that $\tilde{\phi}^{-s\theta_{0}i}({\tilde{w}_1})$,
$s\in[0,1]$ is a positive reparameterization of $${\bar{h}}(\phi^{-s2\pi i}({{\bar{h}}}^{-1}({\tilde{w}_1}))),\,
s\in[0,1].$$ Then $${\tilde{w}_0}:=\tilde{\phi}^{-\theta_{0}i}({\tilde{w}_1})={{\bar{h}}}(\phi^{-2\pi i}({{\bar{h}}}^{-1}({\tilde{w}_1})))\in\mathcal{A}.$$
Let $\gamma:[0,1]\mapsto\mathcal{A}$ be any path such that $\gamma(0)={\tilde{w}_0}$
and $\gamma(1)={\tilde{w}_1}$. For all $t\in[0,1]$ define the paths
 $$\gamma_{t}(s)=\tilde{\phi}^{t\sigma\circ\gamma(s)i}(\gamma(s)),\, s\in[0,1]$$
and \[
\alpha_{t}(s)=\tilde{\phi}^{[(1-s)t\sigma({\tilde{w}_1})+s(t\sigma({\tilde{w}_0})-\theta_{0})]i}({\tilde{w}_1}),\, s\in[0,1].\]
  The paths $\alpha_{t}*\gamma_{t}$ are closed and
give a homotopy between $\alpha_{0}*\gamma$ and $\alpha_{1}*\gamma_{1}$.
By the definition of $\theta_{0}$, the path $\alpha_{0}$ is homotopic
in $\widetilde{T}$ to the path ${{\bar{h}}}\tilde{\phi}^{-s2\pi i}({\bar{h}}^{-1}({\tilde{w}_1}))$,
$s\in[0,1]$. Then $\alpha_{0}*\gamma$ is homotopic to the path ${{\bar{h}}}(\widetilde{\alpha}*\widetilde{\gamma})$,
where $$\widetilde{\alpha}(s)=\tilde{\phi}^{-s2\pi i}({{\bar{h}}}^{-1}({\tilde{w}_1})),\,
s\in[0,1]$$ and $\widetilde{\gamma}={{\bar{h}}}^{-1}\circ\gamma$. But the
path $\widetilde{\alpha}*\widetilde{\gamma}$ is homotopic to $-C$
in $\overline{T}$. Then, it follows from the definition of $g$ that
$\alpha_{0}*\gamma$ is homotopic to $g(-C)$ in $\overline{\widetilde{T}}$.
Therefore $\alpha_{1}*\gamma_{1}$ is homotopic to $g(-C)$ in $\overline{\widetilde{T}}$.
Observe that, since $\gamma_{1}\subset\widetilde{D}$, the path $\alpha_{1}*\gamma_{1}$
is homotopic in $\overline{\widetilde{T}}$ to the closed path $$\tilde{\phi}^{[(1-s)\sigma({\tilde{w}_1})+s(\sigma({\tilde{w}_0})-\theta_{0})]i}(\tilde{\zeta}_{0}),\,
s\in[0,1].$$ Then $g(-C)$ is homotopic to $$\tilde{\phi}^{s(\sigma({\tilde{w}_0})-\sigma({\tilde{w}_1})-\theta_{0})i}(q),\,
s\in[0,1],$$ where $q=\tilde{\phi}^{\sigma({\tilde{w}_1})i}(\tilde{\zeta}_{0})$.
On the other hand, since $\vartheta(\zeta(\tilde{w}))\rightarrow2\pi$
as $\tilde{w}\rightarrow{\tilde{w}_1}$ with $\pi<\vartheta(\zeta(\tilde{w}))<2\pi$,
it follows from the definition of ${\tilde{\vartheta}}$ that ${\tilde{\vartheta}}(\zeta(\tilde{w}))\rightarrow\xi 2\pi $,
where $\xi\in\{1,-1\}$ is such that $\tilde{\phi}^{-s\xi2\pi i}(\tilde{\zeta}_{0})$,
$s\in[0,1]$ is a positive reparameterization of $$g(-C)=\{g\tilde{\phi}^{-s2\pi i}(\zeta_{0}),\,
s\in[0,1]\}.$$ Then $g(-C)$ is homotopic to $$\{\tilde{\phi}^{-s\xi2\pi i}(\tilde{\zeta}_{0}),s\in[0,1]\}=\{\tilde{\phi}^{-s\xi2\pi i}(q),
s\in[0,1]\}.$$ It follows that the paths $\{\tilde{\phi}^{s(\sigma({\tilde{w}_0})-\sigma({\tilde{w}_1})-\theta_{0})i}(q)\}$
and $\{\tilde{\phi}^{-s\xi2\pi i}(q)\}$ are homotopic in $\overline{\widetilde{T}}$
and this implies that \[
\xi 2\pi=-\sigma({\tilde{w}_0})+\sigma({\tilde{w}_1})+\theta_{0}.\]
 Thus, if $\tilde{w}\rightarrow{\tilde{w}_1}$ with $\pi<\vartheta(\zeta(\tilde{w}))<2\pi$,
we have that $\theta(\tilde{w})\rightarrow\theta_{0}$, $\sigma\tilde{\phi}^{-\theta(\tilde{w})}(\tilde{w})\rightarrow\sigma\tilde{\phi}^{-\theta_{0}}({\tilde{w}_1})=\sigma({\tilde{w}_0})$,
${\tilde{\vartheta}}(\zeta(\tilde{w}))\rightarrow\xi 2\pi=-\sigma({\tilde{w}_0})+\sigma({\tilde{w}_1})+\theta_{0}$
and replacing in (\ref{tau}) we obtain that $\sigma(\tilde{w})\rightarrow\sigma({\tilde{w}_1})$.
Therefore $\sigma$ is continuous. On the other hand it is easy to
see that $\tilde{\phi}^{s\sigma(\tilde{w})}({\tilde{w}})\in\widetilde{T}$ for
all $\tilde{w}\in\widetilde{T}_{1}$, $s\in[0,1]$. Assertion is proved.

Define the map \begin{align*}
f & :\widetilde{T}_{1}\mapsto\widetilde{T}\\
f({\tilde{w}}) & =\tilde{\phi}^{\sigma(\tilde{w})}(\tilde{w}).\end{align*}

This map $f$ is an extension of the map $f:\mathcal{A}\rightarrow{\widetilde{D}}$ given by Proposition \ref{redressing}.
 Given $\zeta=(t_{\zeta},0)\in C$, put $g(\zeta)=({\tilde{t}}_{\zeta},0)$
and define the sets \begin{align*}
\mathcal{A}_{\zeta} & ={\bar{h}}(\{(t_{\zeta},x):0<|x|<\delta_1\})\\
{\widetilde{D}}_{\zeta} & =\{({\tilde{t}}_{\zeta},x):|x|<\tilde{\delta}_0\}.\end{align*}

Observe that $f({\tilde{w}})\in{\widetilde{D}}_{\zeta}$ for all ${\tilde{w}}\in\mathcal{A}_{\zeta}\cap {\widetilde{T}_1}$.
 Moreover, the map $$f_{\zeta}=f|_{\mathcal{A}_{\zeta}\cap {\widetilde{T}_1}}:\mathcal{A}_{\zeta}\cap {\widetilde{T}_1}\mapsto{\widetilde{D}}_{\zeta}$$
 can be expressed
as $f_{\zeta}={\tilde{\rho}}f_{0}{\bar{h}}{\rho}{\bar{h}}^{-1}$, where ${\rho}(w)={\phi}^{-\vartheta(\zeta)i}(w)$,
${\tilde{\rho}}(w)=\tilde{\phi}^{{\tilde{\vartheta}}(\zeta)i}(w)$ and $f_0=f|_{\mathcal{A}\cap {\widetilde{T}_1}}$. Clearly ${\rho}$
and ${\tilde{\rho}}$ are diffeomorphisms and by Proposition
\ref{redressing} the map $f_0$ is a homeomorphism. Then $f_{\zeta}$
is a homeomorphism onto its image and $f_{\zeta}({\tilde{w}})$ tends to the
divisor as ${\tilde{w}}$ tends to the divisor. Then  we have the following:
\begin{enumerate}
\item $f$ is a homeomorphism onto its image
\item $f$ preserves the 1-foliation induced by $\widetilde{\mathcal{F}}$   on $\widetilde{T}$
\item $f({\tilde{w}})$ tends to the divisor as ${\tilde{w}}$ tends to the divisor
\item $f$ maps $\mathcal{A}_{\zeta}\cap {\widetilde{T}_1}$ into the Hopf fiber ${\widetilde{D}}_{\zeta}$. Thus, for some $\delta_{2}>0$ we have that $f\circ {\bar{h}}$ maps each Hopf fiber
$\{(t_{\zeta},x):0<|x|<\delta_{2}\}$ into the Hopf fiber $({\tilde{t}}_{\zeta},x):0<|x|<\tilde{\delta}_0\}$.\\

\end{enumerate}

\noindent\textbf{Step 2.} Now,  we will extend $f$ as a topological equivalence of $\widetilde{\mathcal{F}}$
with itself. Thus we will define $\bar{\bar{h}}=f\circ\bar{h}$, which clearly will finish the proof of Theorem \ref{main result}.
First, for some $\varepsilon,\tilde{\delta}_2>0$ we will extend $f$ to the
set $$\mathcal{T}=\{({{t}},x):\sqrt{r_1r_2}-\varepsilon\leq|{{t}}|\leq \sqrt{r_1r_2}+\varepsilon,0<|x|\le \tilde{\delta}_2\}$$ in such way $f=\mbox{id}$ on the boundary
$$\partial\mathcal{T}=\{({{t}},x):|{{t}}|= \sqrt{r_1r_2}\pm\varepsilon,0<|x|\le \tilde{\delta}_2\}.$$  Take first any $\tilde{\delta_2}\in(0,\tilde{\delta}_1)$.
There is $\varepsilon>0$ with the following property: If $\tilde{w}\in \mathcal{T}$, there exists a unique $\tau(\tilde{w})\in\mathbb{R}$ such that $\tilde{\phi}^{\tau(\tilde{w})}(\tilde{w})$ is contained in $\widetilde{T}_1$. Put $\varrho(\tilde{w})=\tilde{\phi}^{\tau(\tilde{w})}(\tilde{w})$. Take a continuous function $${\nu}:[\sqrt{r_1r_2}-\varepsilon,\sqrt{r_1r_2}+\varepsilon]\mapsto[0,1]$$ such
that ${\nu}(\sqrt{r_1r_2}\pm\varepsilon)=0$ and ${\nu}(\sqrt{r_1r_2})=1$.
Then denote ${\tilde{w}}=(\tilde{t},\tilde{x})$ and extend $\sigma$ and $f$ by the expressions:\begin{align*}
\sigma({\tilde{w}}) &=    {\nu}(|\tilde{t}|)\sigma(\varrho(\tilde{w}))\\
f({\tilde{w}}) & =    \tilde{\phi}^{\sigma({\tilde{w}})}({\tilde{w}}).\end{align*}
Naturally we extend $f$ as the identity map outside $\mathcal{T}$.  Then  it is not difficult to see that close to the exceptional divisor $f$  is a topological equivalence of $\widetilde{\mathcal{F}}$
with itself.

\section{\label{redressing nh} Proof of Proposition \ref{redressing} in the non-hyperbolic case}
In this section we assume that the holonomy of  $\widetilde{\mathcal{F}}|_{\overline{\widetilde{V}}}$ at the leaf $\{x=0\}$ is non-hyperbolic.
 Fix a punctured Hopf fiber
${D_0}^*=\{(t_{0},x):0<|x|<\delta_{0}\}$ with $|t_{0}|=r_{12}$. Clearly in general
the set ${\mathcal{A}}_0=h({{D_0}^*})$ is not a Hopf fiber, but we can use the flow to move each point of ${\mathcal{A}}_0=h({{D_0}^*})$ the correct amount to obtain a Hopf fiber.   Given $\tilde{w}_1,\tilde{w}_2\in{\mathcal{A}}_0$, take a continuous path $\gamma=(\gamma_{1},\gamma_{2})$
in ${\mathcal{A}}_0$ with $\gamma(0)=\widetilde{w}_{1}$ and $\gamma(1)=\widetilde{w}_2$.
We can write $\gamma_{1}(s)=r(s)e^{i\alpha(s)}$ for some continuous
functions $r$ and $\alpha$. Define $\Theta(\widetilde{w}_1,\widetilde{w}_2)=\alpha(1)$.
From Proposition \ref{homologico basico} we see that $\Theta(\widetilde{w}_1,\widetilde{w}_2)$
does not depend on the path $\gamma$. The function $\Theta$ measures the oscillation of ${\mathcal{A}}_0$ around $\{t=0\}$, that is,
the deviation of ${\mathcal{A}}_0$ from being a Hopf fiber.
\begin{rem}It is easy to see from the definition of $\Theta$ that $$\Theta(\widetilde{w}_1,\widetilde{w}_2)+\Theta(\widetilde{w}_2,\widetilde{w}_3)=\Theta(\widetilde{w}_1,\widetilde{w}_3)$$
for any $\widetilde{w}_1,\widetilde{w}_2,\widetilde{w}_3\in{\mathcal{A}}_0$. From this and the mean value theorem we deduce that,
if $\Theta(\widetilde{w}_1,\widetilde{w}_2)=\epsilon_1+\ldots+\epsilon_k$ with $\epsilon_j>0$ and $C$ is a simple path connecting $\widetilde{w}_1$ and $\widetilde{w}_2$, then there are intermediary ordered points $\mathfrak{w}_0,\ldots,\mathfrak{w}_k$ in $C$ with $\mathfrak{w}_0=\widetilde{w}_1$, $\mathfrak{w}_k=\widetilde{w}_2$ and such that $|\Theta(\mathfrak{w}_{j-1},\mathfrak{w}_{j})|=\epsilon_j$.
\end{rem}
Proposition \ref{redressing} is based in the following proposition.
\begin{prop}
\label{control en teta} There exists $\delta_{1}\in(0,\delta_{0})$ and $\mu_3>0$
such that the set $${{\mathcal{A}_1}}=h(\{(t_{0},x):0<|x|<\delta_{1}\})$$ has the following property: for  any $\widetilde{w}_1,\widetilde{w}_2\in{{\mathcal{A}_1}}$, $\tilde{w}_1=(\tilde{t}_1,\tilde{x}_1)$  we
have that $\widetilde{\phi}^{is}(\widetilde{w_1})$ is well defined and $|\widetilde{\phi}^{is}_2(\tilde{w}_1)|\le\mu_3\sqrt{|\tilde{x}_1|}$ for  $|s|\le |\Theta(\widetilde{w}_1,\widetilde{w}_2)|$.
\end{prop}

This section is organized as follows. We start with a lemma  and the intermediary propositions \ref{parte}, \ref{esa}, \ref{lem6} and \ref{lem7}. In Subsection \ref{nonhyperbolic} we prove a preparatory proposition and give the proof of Proposition \ref{control en teta}. Finally, in Subsection \ref{sec redressing} we proof Proposition \ref{redressing}.  It is worth mentioning that Propositions \ref{lem6} and \ref{lem7}  are independent of the Lipschitz hypothesis.

 \begin{lem}\label{coca} Let $\widetilde{w}_1$ and $\widetilde{w}_2$ be points in ${\mathcal{A}}_0$  such that
$$|\Theta(\widetilde{w}_1,\widetilde{w}_2)|=\pi\in\mathbb{R}.$$ Let $z_1, z_2\in {\mathbb{C}}^2$ be such that
$\mathfrak{h}(z_j)=\pi(\widetilde{w}_j)$, $j=1,2$. Then
$$|z_1-z_2|\geq\frac{r_1m}{LM}|z_2|.$$
  \end{lem}
\begin{proof}
Since $|\Theta(\widetilde{w}_{1},\widetilde{w}_{2})|=\pi$,
it is easy to see that $|\widetilde{w}_{1}-\widetilde{w}_{2}|\ge2r_{1}/L.$
Then, if $\mathfrak{h}(z_{2})=(\widetilde{x}_{2},\widetilde{y}_{2})$,
by Lemma $\ref{bu1}$ we have \[
|z_{1}-z_{2}|\ge\frac{1}{M}|\mathfrak{h}(z_{1})-\mathfrak{h}(z_{2})|\ge\frac{1}{M}|\widetilde{x}_{2}||\widetilde{w}_{1}-\widetilde{w}_{2}|\ge\frac{2r_{1}}{LM}|\widetilde{x}_{2}|.\]
But $|\mathfrak{h}(z_{2})|=|\widetilde{x}_{2}|+|\widetilde{y}_{2}|\le|\widetilde{x}_{2}|+Lr_{2}|\widetilde{x}_{2}|\le2|\widetilde{x}_{2}|$,
hence\[
|z_{1}-z_{2}|\ge\frac{2r_{1}}{LM}|\widetilde{x}_{2}|\ge\frac{2r_{1}}{LM}\frac{|\mathfrak{h}(z_{2})|}{2}\ge\frac{r_{1}m}{LM}|z_{2}|.\]

\end{proof}

 \begin{prop}\label{parte}Let $(t_0,x_1),(t_0,x_2)$   be points in ${{D_0}^*}$ and denote $\widetilde{w}_1=(\widetilde{t}_1,\widetilde{x}_1)=h(t_0,x_1)$ and $\widetilde{w}_2=(\widetilde{t}_2,\widetilde{x}_2)=h(t_0,x_2)$. Then  there exists a constant $\mu_1> 0$
with the following property. If
$x_2=\lambda x_1$ with $0<\lambda\le 1$ and
$$|\Theta(\widetilde{w}_1,\widetilde{w}_2)|= 2n\pi\quad \textrm{for some}\quad n\in\mathbb{N},$$
 then
$$|\widetilde{x}_2|\leq \frac{\frac{2M}{m}|\widetilde{x}_1|}{(1+\mu_1)^n}.$$

\end{prop}
\proof Consider the segment $$L=\{(sx_1,st_0x_1): \lambda\leq s\leq
1\}$$ in ${\mathbb{C}}^2\backslash\{0\}$.
Clearly $z_1=(x_1,t_0x_1)$ and $z_2=(\lambda x_1,\lambda t_0x_1)$ are the
endpoints of $L$ and $\widetilde{w}_1=h\circ\pi^{-1}(z_1)$, $\widetilde{w}_2=h\circ\pi^{-1}(z_2)$. Since the simple path $h\circ\pi^{-1}(L)$ connects $\widetilde{w}_1$ with $\widetilde{w}_2$ and $|\Theta(\widetilde{w}_1,\widetilde{w}_2)|= 2n\pi$, we can find ordered points  $$\mathfrak{w}_0=h\circ\pi^{-1}(\mathfrak{z}_0),\ldots,\mathfrak{w}_{2n}=h\circ\pi^{-1}(\mathfrak{z}_{2n})$$ in $h\circ\pi^{-1}(L)$, with $\mathfrak{z}_0=z_1$, $\mathfrak{z}_{2n}=z_2$
 such that $$|\Theta(\mathfrak{w}_{j},\mathfrak{w}_{j+1})|=\pi$$ for all $j=0,\ldots,2n-1.$ Then
 by
Lemma \ref{coca} we have that
$$|\mathfrak{z}_j-\mathfrak{z}_{j+1}|\geq \frac{r_1m}{LM}|\mathfrak{z}_j|.$$
 Observe that  $|\mathfrak{z}_j-\mathfrak{z}_{j+1}|=|\mathfrak{z}_j|-|\mathfrak{z}_{j+1}|$, because $\mathfrak{z}_j$ and $\mathfrak{z}_{j+1}$ are contained in the segment $L$. Then
 $|\mathfrak{z}_j|-|\mathfrak{z}_{j+1}|\geq \frac{r_1m}{LM}|\mathfrak{z}_{j+1}|$, so
\begin{equation}\label{ajo} |\mathfrak{z}_{j+1}|\leq\frac{|\mathfrak{z}_j|}{1+\frac{r_1m}{LM}}\nonumber\end{equation}
and it follows  that
$$|\mathfrak{z}_{j}|\leq\frac{|\mathfrak{z}_0|}{(1+\frac{r_1m}{LM})^j}$$ for all $j=1,\ldots,2n$.
In particular, we have that
\begin{equation}
\label{pato}|z_2|\leq \frac{|z_1|}{(1+\frac{r_1m}{LM})^{2n}}.\end{equation}
  Recall that
$$M|z_2|\geq
|\mathfrak{h}(z_2)|=|\pi(\widetilde{w}_2)|=|(\widetilde{x}_2,\widetilde{t}_2\widetilde{x}_2)|>|\widetilde{x}_2|,$$ hence
\begin{equation}\label{pollo}(1/M)|\widetilde{x}_2|<|z_2|.\end{equation}
On the other hand:
$$m|z_1|\leq|\mathfrak{h}(z_1)|=|\pi(\widetilde{w}_1)|=|(\widetilde{x}_1,\widetilde{t}_1\widetilde{x}_1)|=|\widetilde{x}_1|+|\widetilde{t}_1||\widetilde{x}_1|<2|\widetilde{x}_1|$$
   and we obtain
\begin{equation}\label{perro}  |z_1|\leq (2/m)|\widetilde{x}_1|.\end{equation}
From \ref{pollo}, \ref{pato} and \ref{perro}:
$$(1/M)|\widetilde{x}_2|\leq |z_2|\leq \frac{|z_1|}{(1+\frac{r_1m}{LM})^{2n}}\leq \frac{(2/m)|\widetilde{x}_1|}{(1+\frac{r_1m}{LM})^{2n}},$$
hence $$|\widetilde{x}_2|\leq \frac{\frac{2M}{m}|\widetilde{x}_1|}{(1+\frac{r_1m}{LM})^{2n}}.$$
Finally,  we take $\mu_1>0$ be such that
$1+\mu_1=(1+\frac{r_1m}{LM})^2$. \qed

\begin{prop}\label{esa} There exists a constant  $K_1>0$ with the following property.
If $ \widetilde{w}_1=h(t_0,x_1)$ and $\widetilde{w}_2=h(t_0,x_2)$ are points in ${\mathcal{A}}_0$ such
that $|x_1|=|x_2|$,  then $$|\Theta(\widetilde{w}_1,\widetilde{w}_2)|\leq 2K_1\pi.$$
\end{prop}
\proof Let $|x_1|=|x_2|=\rho$ and let $S\subset{\mathbb{C}}^2$ be the circle
$\{(\zeta,t_0\zeta): \zeta\in{\mathbb{C}}, |\zeta|=\rho\}$.  Let   $C\subset S$ be
an arc of  $S$ joining $z_1=(x_1,t_0x_1)$ and $z_2=(x_2,t_0x_2)$. Suppose that for some
$n\in\mathbb{N}$ we have $$|\Theta(\widetilde{w}_1-\widetilde{w}_2)|>2n\pi.$$ Then as in the proof of Proposition \ref{parte} there are ordered points $\mathfrak{z}_0<...<\mathfrak{z}_{2n}$ in $C$
with $\mathfrak{z}_0=z_1$, $\mathfrak{z}_{2n}\leq
z_2$ and  such that, if $\mathfrak{w}_j={h}\circ\pi^{-1}(\mathfrak{z}_j)$, then
$$|\Theta(\mathfrak{w}_j,\mathfrak{w}_{j+1})|=\pi$$
for all $j=0,\ldots,2n-1$. Then by Lemma
\ref{coca} we have
\begin{equation}\label{pico}|\mathfrak{z}_j-\mathfrak{z}_{j+1}|\geq \frac{r_1m}{LM}|\mathfrak{z}_{j+1}|\end{equation}
for all $z=0,\ldots,2n-1$. Since $\mathfrak{z}_{j}\in C$, then
$\mathfrak{z}_{j}=(\zeta_j, t_0\zeta_j)$ with $|\zeta_j|=\rho$ and therefore
$$|\mathfrak{z}_j-\mathfrak{z}_{j+1}|=|(\zeta_j-\zeta_{j+1},t_0(\zeta_j-\zeta_{j+1}))|=
|(1,t_0)||\zeta_j-\zeta_{j+1}|,$$
 $$|\mathfrak{z}_{j+1}|=|(\zeta_{j+1},t_0\zeta_{j+1})|=|(1,t_0)||\mathfrak{z}_{j+1}|=|(1,t_0)|\rho.$$
Thus, from \ref{pico} we have
$$|(1,t_0)||\zeta_j-\zeta_{j+1}|\geq \frac{r_1m}{LM}|(1,t_0)|
\rho,$$
$$|\zeta_j-\zeta_{j+1}|\geq \frac{r_1m}{LM}\rho,$$
hence
$$\sum_{j=0}^{
2n-1}|\zeta_j-\zeta_{j+1}|\geq 2n\frac{r_1m}{LM}\rho.$$
 Observe
that the sum above is less or equal than the length
  of the circle
$|\zeta|=\rho$ in ${\mathbb{C}}$, that is, $2\pi \rho$. Then
$$2n\frac{r_1m}{LM}\rho\le 2\pi \rho,$$
and so $$n\leq \frac{\pi LM}{r_1m}.$$ Thus, it is
sufficient to take an integer
$K_1> \frac{\pi LM}{r_1m}.$\qed

\begin{prop}\label{lem6}  There exists a constant $s_0>0$ such that for any $w=(t,x)\in \widetilde{V}$ with  $|x|\le \widetilde{\delta}_0/2$ we have that   $\widetilde{\phi}^{si}(w)$ is defined for $s\in [-s_0,s_0]$ and
$|\widetilde{\phi}_2^{si}(w)|\le 2|x|$ for all $s\in [-s_0,s_0]$.
\end{prop}
\proof  Denote $\widetilde{\phi}^{si}(w)=(\tau(s),\alpha(s))$. If we set $\eta_1=\eta L r_2/2$ from Fact \ref{hecho} we have that $|\alpha'(s)|\le \eta_1|\alpha(s)|$. Define  $f(s)=|\alpha(s)|^2$. Then
$$ |f'(s)|\leq 2|\alpha'(s)||{\alpha}(s)|\le 2\eta_1 |{\alpha}(s)|^2,$$
$$|f'(s)|\leq 2\eta_1 f(s).$$
 Thus, if $\widetilde{\phi}^{si}(w)$ is defined and contained in $\widetilde{V}$ for all $s$  in an interval $I$, then we have that   $f(s)\leq f(0)e^{2\eta_1|s|}$ for all $s\in I$, hence $|\alpha(s)|\leq |x|e^{\eta_1|s|}$  for all $s\in I$. If $S>0$ is maximal such that $\widetilde{\phi}^{si}(w)$ is defined on $(-S,S)$ and contained in $\widetilde{V}$, then, either  $S=\infty$,
 or there exists a sequence $(s_k)$ with $|s_k|\rightarrow S$ such that   $|\alpha(s_k)|$ tends to $\widetilde{\delta}_0$. But $\lim |\alpha(s_k)|\leq \lim |x|e^{\eta_1|s_k|}$ as $s\rightarrow S$, hence
$\widetilde{\delta}_0\leq |x|e^{\eta_1 S}\le  (\widetilde{\delta}_0/2)e^{\eta_1 S}$ and therefore $S\ge \frac{\log2}{\eta_1}$. Finally it is sufficient to take $s_0<\frac{\log2}{\eta_1}$.\qed

By successively applications of Proposition \ref{lem6} we obtain:
 \begin{prop}\label{lem7} If  $w=(t,x)\in \widetilde{V}$ is such that $|x|\le\widetilde{\delta}_0/2^n$ with $n\in \mathbb{N}$, then   $\widetilde{\phi}^{si}(w)$ is defined for $s\in [-ns_0,ns_0]$ and
$|\widetilde{\phi}_2^{si}(w)|\le 2^n|x|$ for all $s\in [-ns_0,ns_0]$.
 \end{prop}

 \subsection{\label{nonhyperbolic} Proof of Proposition \ref{control en teta}.} We start with a preparatory proposition.

\begin{prop}\label{lem8} There exist $\mu_2,\mu_3>0$ with the following property. If  $n \in \mathbb{N}$ and
  $w=(t,x)\in{\mathcal{A}}_0$ is such that $|x|\le\frac{\mu_2}{(1+\mu_1)^n}$,
   then $\widetilde{\phi}^{si}(w)$ is defined and $$|\widetilde{\phi}^{si}_2(w)|\le\mu_3\sqrt{|x|}$$ for  all $s\in[-2n\pi,2n\pi]$
\end{prop}
\proof The holonomy of $\widetilde{\mathcal{F}}|_{\widetilde{V}}$ at the leaf $\{x=0\}$ is represented by the maps $$x\mapsto \widetilde{\phi}_2^{2\pi i}(t,x),$$ where  $L_1r_1\le |t|\le L_2r_2$. Let $\lambda$ be the multiplier of those maps at $x=0$. We can find ${\tilde{\delta}_3}\in(0,\widetilde{\delta}_0)$ such that, if $|x|\le{\tilde{\delta}_3}$ and $L_1r_1\le |t|\le L_2r_2$, then $$|\widetilde{\phi}_2^{ 2\pi i}(t,x)-\lambda x|\le(\sqrt{1+\mu_1}-1)|x|$$  and
$$|\widetilde{\phi}_2^{ -2\pi i}(t,x)-\lambda^{-1} x|\le(\sqrt{1+\mu_1}-1)|x|.$$ In this case,
 since $|\lambda|=1$ we have $$|\widetilde{\phi}_2^{\pm 2\pi i}(t,x)|\leq \sqrt{1+\mu_1}|x|.$$ Take a natural $n_0\ge \frac{2\pi}{s_0}$ and define $\mu_2=\frac{{\tilde{\delta}_3}}{2^{n_0}}$. We prove first that $|x|\le\frac{\mu_2 }{(1+\mu_1 )^{n}}$ implies that $\widetilde{\phi}^{is}(w)$ is defined for all $s\in[-2n\pi,2n\pi]$.
 We proceed by induction.
Suppose $|x|\le\frac{\mu_2}{1+\mu_1}$.  Then
$|x|\le\frac{{\tilde{\delta}_3}}{2^{n_0}}$ and by Proposition \ref{lem7} we have that
$\widetilde{\phi}^{is}(t,x)$ is defined for all $s\in [-n_0s_0,n_0s_0]\supset[-2\pi,2\pi]$, which proves case $n=1$.
Now, suppose that $|x|\le\frac{\mu_2 }{(1+\mu_1 )^{n}}$ implies that $\widetilde{\phi}^{is}(w)$ is defined for all $s\in[-2n\pi,2n\pi]$  and let $w=(t,x)$ be such that
$|x|\le\frac{\mu_2 }{(1+\mu_1 )^{n+1}}$. Then
$|x|\le\frac{\mu_2 }{(1+\mu_1 )^{n}}$ and by induction
hypothesis we have that $\widetilde{\phi}^{is}(w)$ is defined for all $s\in[-2n\pi,2n\pi]$.
Since $|x|\le{\tilde{\delta}_3}$ we have
$$|\widetilde{\phi}^{2\pi i}_2(w)|\le\sqrt{1+\mu_1} |x|\le\frac{\mu_2 }{(1+\mu_1 )^{n+\frac{1}{2}}}\le{\tilde{\delta}_3}.$$
Again, since $|\widetilde{\phi}^{2\pi i}_2(w)|\le{\tilde{\delta}_3}$ we have
$$|\widetilde{\phi}^{4\pi i}_2(w)|=|\widetilde{\phi}^{2\pi i}_2(\widetilde{\phi}^{2\pi i}(w))|\le\sqrt{1+\mu_1}|\widetilde{\phi}^{2\pi i}_2(t,x)|\le \sqrt{1+\mu_1}^2|x|     \le \frac{\mu_2 }{(1+\mu_1 )^{n}}\le{\tilde{\delta}_3}.$$
If we proceed successively we obtain
$$|\widetilde{\phi}^{2n\pi i}_2(w)|\le\sqrt{1+\mu_1}^n|x| \le\frac{\mu_2 }{(1+\mu_1 )^{\frac{n}{2}+1}}\le{\tilde{\delta}_3}.$$
Then, if
$w_1=\widetilde{\phi}^{2n\pi i}(w)$, by  case $n=1$ we have that $\widetilde{\phi}^{is}(w_1)$ is defined for all $s\in[-2\pi,2\pi]$.
From this and the flow property we deduce that $\widetilde{\phi}^{is}(w)$ is defined for all $s\in[-2n\pi,2(n+1)\pi]$. In the same way we prove that $\widetilde{\phi}^{is}(w)$ is defined for all $s\in[-2(n+1)\pi,2n\pi]$. On the other hand, there is a constant $c>0$ such that $|\widetilde{\phi}_2^{is}(t,x)|\le c|x|$ for all $s\in [-2\pi,2\pi]$, $|x|\le\tilde{\delta}_3$. Suppose $|x|\le\frac{\mu_2 }{(1+\mu_1 )^{n}}$ and let $s\in[0,2n\pi]$. The case  $s\in[-2n\pi,0]$ is similar. Put $s=2\pi k+\xi$  with $k\in\mathbb{Z}$ and $\xi\in[0,2\pi)$. Then
$$|\tilde{\phi}_2^{si}(w)|\le |\tilde{\phi}_2^{\xi i}\tilde{\phi}_2^{2\pi k i}(w)|\le c|\tilde{\phi}_2^{2\pi k i}(w)|\le c\sqrt{1+\mu_1}^k|x|.$$
But $|x|\le\frac{\mu_2 }{(1+\mu_1 )^{n}}$ implies  $$\sqrt{1+\mu_1}^k\le \sqrt{1+\mu_1}^n\le\frac{\sqrt{\mu_2}}{\sqrt{|x|}},$$ and therefore $$|\tilde{\phi}_2^{si}(w)|\le c\sqrt{\mu_2}\sqrt{|x|},$$ hence we set $\mu_3=c\sqrt{\mu_2}$.

\noindent\textit{Proof of Proposition \ref{control en teta}}.
Take $\delta_1\in(0,\delta_0)$ such that ${{\mathcal{A}_1}}$ is contained in the set $$\{(t,x)\in\widetilde{V}:|x|\le\frac{m\mu_2}{2M(1+\mu_1)^{K_1+1}}\}.$$
Let $\widetilde{w}_1=h(w_1),\widetilde{w}_2=h(w_2)\in{{\mathcal{A}_1}}$, put $w_1 =(t_0,x_1)$, $w_2=(t_0,x_2)$, $\widetilde{w}_1=(\tilde{t}_1,\tilde{x}_1)$, $\widetilde{w}_2=(\tilde{t}_2,\tilde{x}_2)$ and assume $|x_1|\le |x_2|$.
 Take $w'=(t_0,x')$ with $|x'|=|x_2|$ and such that
$x_1=\lambda x'$ for $\lambda \in (0,1]$.
Since $$|\tilde{x}_1|\le\frac{m\mu_2}{2M(1+\mu_1)^{K_1+1}}\le\frac{\mu_2}{(1+\mu_1)^{K_1+1}},$$
 by Proposition \ref{lem8} we have that
 $\widetilde{\phi}^{is}(w_1)$ is defined for $|s|\le 2(K_1+1)\pi$. Then we can assume $|\Theta(\widetilde{w}_1,\widetilde{w}_2)|>2(K_1+1)\pi$. But $\Theta(\widetilde{w}_1,\widetilde{w}_2)=\Theta(\widetilde{w}_1,\widetilde{w}')+\Theta(\widetilde{w}',\widetilde{w}_2)$
 and $|\Theta(\widetilde{w}',\widetilde{w}_2)|\le 2K_1\pi$ by Proposition \ref{esa}, so we have $|\Theta(\widetilde{w}_1,\widetilde{w}')|=2\kappa\pi$ with $\kappa>1$.
Let $n\in\mathbb{N}$
be such that  $n\leq\kappa<n+1$. It follows from the mean
value theorem that there exists $w''=(t_0,x'')$ with $x''$ in the segment
joining $x_1$ and $x'$ such that $|\Theta(\widetilde{w}_1,\widetilde{w}'')|=2n\pi$, where $\widetilde{w}''=(\tilde{t}'',\tilde{x}'')=h(w'')$.
Then, by Proposition \ref{parte} we have $$|\tilde{x}_1|\leq
\frac{(2M/m)|\tilde{x}''|}{(1+\mu_1)^n}.$$ Since $|x''|\in\mathcal{A}_1$ we have
that
$$|\tilde{x}''|\le\frac{m\mu_2}{2M(1+\mu_1)^{K_1+1}}.$$
Then
$$|\tilde{x}_1|\le\frac{(2M/m)}{(1+{\mu}_1)^n}\frac{m{\mu}_2}{2M(1+{\mu}_1)^{K_1+1}}=\frac{{\mu}_2}{(1+{\mu}_1)^{n+K_1+1}},$$
so
by Proposition \ref{lem8} we have that $\widetilde{\phi}^{is}({\tilde{w}}_1)$ is defined and $|\widetilde{\phi}^{is}_2(\tilde{w}_1)|\le\mu_3\sqrt{|x|}$ for
$|s|\le 2(n+K_1+1)\pi$.
   Finally, it suffices to observe that
   $$|\Theta({\tilde{w}}_1,{\tilde{w}}_2)|\le |\Theta({\tilde{w}}_1,{\tilde{w}}')|+|\Theta({\tilde{w}}',{\tilde{w}}_2)|\leq 2\kappa\pi+ 2K_1\pi\le 2(n+K_1+1)\pi.$$

   \subsection{\label{sec redressing}Proof of Proposition \ref{redressing}}

  Set $$D^{*}=\{(t_{0},x):0<|x|<\delta_{1}\},$$
fix $\tilde{w}_{2}\in\mathcal{A}_{1}=h(D^*)$ and define  $\sigma_{1}(\tilde{w}_{1})=\Theta(\tilde{w}_{1},\tilde{w}_{2})$ for all $\tilde{w}_{1}\in\mathcal{A}_{1}$. By Proposition \ref{control en teta} we have that $\tilde{\phi}^{s i}(\tilde{w}_{1})$
is defined  and $|\tilde{\phi}^{s i}_2(\tilde{w}_{1})|\le \mu_3\sqrt{|\tilde{x}_1}|$ for $|s|\le|\sigma_1(\tilde{w}_1)|$. Clearly
we can assume $\delta_1$ to be  small enough such that, for each $\varsigma\in\mathbb{S}^1$, the radial segment $$\{(s\varsigma ,0):\frac{1}{L}r_{1}\le s\le Lr_{2}\}$$
can be lifted to the leaf of $\mathcal{F}|_{\widetilde{V}}$ through
any $\tilde{\phi}^{s i}(\tilde{w}_{1})$, $|s|\le|\sigma_1(\tilde{w}_1)|$. Then, if $\tilde{w}_{1}=(\tilde{t}_{1},\tilde{x}_{1})$,
we have that $\tilde{\phi}^{T}(\tilde{w}_{1})$ is defined for all
$T$ in the rectangle \[
\left\{ \tau+\theta i:\ln\frac{r_{1}}{L|\tilde{t}_{1}|}\le\tau\le\ln\frac{Lr_{2}}{|\tilde{t}_{1}|},\,|\theta|\le|\sigma_{1}(\tilde{w}_{1})|\right\}.\]
Set $\mathcal{A}=\bar{h}(D^{*})$. By Proposition \ref{regularizado}
we have:
\begin{enumerate}
\item  $\mathcal{A}\subset\widetilde{T}$
\item Any  $\tilde{w}\in\mathcal{A}$
can be expressed as $\tilde{w}=\tilde{\phi}^{\overline{T}(\tilde{w}_{1})}(\tilde{w}_{1})$
for a unique $\tilde{w}_{1}=(\tilde{t}_{1},\tilde{x}_{1})$  in $\mathcal{A}_{1}$
\item $\mbox{Im}(\overline{T})$ is bounded by the constant $\tilde{\theta}_{1}>0$.
\end{enumerate}
Clearly we can assume that $\tilde{\phi}^{T}(\tilde{w}_{1})$ is defined
for all $T$ in the rectangle\[
\left\{\tau+\theta i:\ln\frac{\frac{1}{L}r_{1}}{|\tilde{t}_{1}|}\le\tau\le\ln\frac{Lr_{2}}{|\tilde{t}_{1}|},\,|\theta|\le\tilde{\theta}_{1}\right\}.\]
Then, for any $\tilde{w}=(\tilde{t},\tilde{x})\in\mathcal{A}$ with
\foreignlanguage{english}{$\tilde{w}=\tilde{\phi}^{\overline{T}(\tilde{w}_{1})}(\tilde{w}_{1})$},
by the flow property we deduce that $\tilde{\phi}^{T}(\tilde{w})$
is defined for all $T$ in the rectangle\[
\{\tau+\theta i:\ln\frac{\frac{1}{L}r_{1}}{|\tilde{t}|}\le\tau\le\ln\frac{Lr_{2}}{|\tilde{t}|},\,
\theta=s[\sigma_{1}(\tilde{w}_{1})-\mbox{Im}\overline{T}(\tilde{w}_1)],\: s\in[0,1]\}.\]
then, if we define $\sigma(\tilde{w})=\sigma_{1}(\tilde{w}_{1})-\mbox{Im}\overline{T}(\tilde{w}_1)$
we see that $\tilde{\phi}^{s\sigma(\tilde{w})i}(\tilde{w})$ is defined
for all $s\in[0,1]$. By the definition of $\Theta$ it is easy to
see that $f(\tilde{w}):=\tilde{\phi}^{\sigma(\tilde{w})i}(\tilde{w})$
is contained in a Hopf fiber $\widetilde{D}=\{(\tilde{t}_{0},x):|x|<\tilde{\delta}_{0}\}$
with $|\tilde{t}_{0}|=\sqrt{r_{1}r_{2}}$. Denote $f(\tilde{w})=(f_1(\tilde{w}),f_2(\tilde{w}))$ and observe that $$f(\tilde{w})=\tilde{\phi}^{\mbox{Re}\overline{T}(\tilde{w}_1)}\tilde{\phi}^{\sigma_1(\tilde{w}_1)i}(\tilde{w}_1).$$
Since $|\mbox{Re}\overline{T}|$ is bounded we find a constant $\mu_4>0$  such that
$$|f_2(\tilde{w})|=|\tilde{\phi}_2^{\mbox{Re}\overline{T}(\tilde{w}_1)}\tilde{\phi}^{\sigma_1(\tilde{w}_1)i}(\tilde{w}_1)|\le \mu_4|\tilde{\phi}^{\sigma_1(\tilde{w}_1)i}(\tilde{w}_1)|\le \mu_4\mu_3\sqrt{|\tilde{x}_1|}.$$
 Then, since $|\tilde{x}_1|\rightarrow 0$ as $|\tilde{x}|\rightarrow 0$, we deduce that
 $f(\tilde{w})\rightarrow o\in\widetilde{D}$ as $\tilde{w}$ tends to the exceptional divisor.
It remains to prove that $f$
is injective.
Suppose that $f({\tilde{w}})=f({\tilde{w}'})$. Let $\gamma:[0,1]\mapsto{\mathcal{A}}$
be a curve joining ${\tilde{w}}$ and ${\tilde{w}'}$. Consider the paths
 $$\alpha(s)={\tilde{\phi}}^{(1-s){\sigma}({\tilde{w}})}({\tilde{w}}),\,s\in[0,1]$$ and $$\beta(s)={\tilde{\phi}}^{s{\sigma}({\tilde{w}'})}({\tilde{w}'}),\, s\in[0,1].$$
Let $\vartheta$ be the closed path $\gamma*\beta*\alpha$.
For $t\in[0,1]$ define the paths $$\gamma_t(s)={\tilde{\phi}}^{t{\sigma}\circ\gamma(s)}(\gamma(s)),\, s\in[0,1],$$ $$\alpha_t(s)={\tilde{\phi}}^{(1-s+ts){\sigma}({\tilde{w}})}({\tilde{w}}),\, s\in[0,1]$$
and $$\beta_t(s)={\tilde{\phi}}^{(s+t(1-s)){\sigma}({\tilde{w}'})}({\tilde{w}'}),\, s\in[0,1].$$ It is easy to see
that $\gamma_{t}*\beta_{t}*\alpha_{t}$ define a homotopy between
$\vartheta$ and a path contained in ${\widetilde{D}}$. Then $\vartheta$ does not
link the set $\{t=0\}$ and therefore, by Proposition \ref{homologico basico},
the path ${\bar{h}}^{-1}(\vartheta)$ does not link $\{t=0\}$ . Observe that
the path ${\bar{h}}^{-1}(\vartheta)$ has the part ${{\bar{h}}}^{-1}(\gamma)$ contained
in ${D^*}$. The other part ${{\bar{h}}}^{-1}(\beta*\alpha)$ is a path contained
in a leaf of the foliation $\mathcal{F}|_V$.
Since ${{\bar{h}}}^{-1}(\beta*\alpha)$ joins ${{\bar{h}}}^{-1}({\tilde{w}})$ and ${{\bar{h}}}^{-1}({\tilde{w}'})$
(points in ${D^*}$) we have that ${{\bar{h}}}^{-1}({\tilde{w}})=g({{\bar{h}}}^{-1}({\tilde{w}'}))$, where
$g$ is the holonomy map associated to the projection of ${\bar{h}}^{-1}(\vartheta)$
in $\{x=0\}$. Then, since ${\bar{h}}^{-1}(\vartheta)$ does not link $\{t=0\}$,
we have that $g=\textrm{id}$, hence ${\tilde{w}}={\tilde{w}'}$.

\section{Proof of Proposition \ref{redressing} in the hyperbolic case}\label{hyperbolic}
In this section we assume that the holonomy of  $\widetilde{\mathcal{F}}|_{\overline{\widetilde{V}}}$ at the leaf $\{x=0\}$ is hyperbolic. Let $D_0^*$ as in section \ref{redressing nh} and put  $\mathcal{A}_0=\bar{h}(D_0^*)$. For $\tilde{w}_1,\tilde{w}_2\in\mathcal{A}_0$ define $\Theta(\tilde{w}_1,\tilde{w}_2)$ as in Section \ref{redressing nh}, fix $\tilde{w}_2\in\mathcal{A}_0$ and define $\sigma(\tilde{w})=\Theta(\tilde{w},\tilde{w}_2)$.
Given $z\in\mathcal{A}_0$ take a complex disc $\Sigma_{z}$ passing
through $z$ and transverse to $\tilde{\mathcal{F}}$. In a neighborhood
$U_{z}$ of $z$ is well defined a leaf preserving projection $\pi_{z}:U_{z}\mapsto\Sigma_{z}$.
It is not difficult to prove, since $\mathcal{A}_0$ is a continuous
transversal to $\tilde{\mathcal{F}}$, that in a small neighborhood $\Delta_{z}$
of $z$ in $\mathcal{A}_0$ the restriction $\pi_{z}:\Delta_{z}\mapsto\Sigma_{z}$
is a homeomorphism onto its image. The charts $\lbrace\pi_{z}\rbrace_{z\in\mathcal{A}_0}$
define a natural complex structure on $\mathcal{A}_0$. Then $\mathcal{A}_0$ is analytically equivalent
to an annulus $$\{z\in\mathbb{C}:0\leq r<|z|\leq1\}$$ for some $r\geq0$.
The holonomy of  $\widetilde{\mathcal{F}}|_{\overline{\widetilde{V}}}$ at the leaf $\{x=0\}$ is represented by a contractive function
$g:{D}_0\mapsto {D}_0$, where $$D_0=\{(t_0,x):|x|<\delta_0\}.$$
Consider the map $\tilde{g}={\bar{h}}\circ g\circ{\bar{h}}^{-1}$.
Clearly $\tilde{g}:\mathcal{A}_0\mapsto\mathcal{A}_0$ is not trivial at homology
level and it is holomorphic, because it is continuous and leaf preserving.
Then, since $g'$ is not an isomorphism, it follows from the annulus
theorem (see \cite{remmert}, p. 211) that $r=0$ and $\mathcal{A}_0$ is therefore
analytically equivalent to a punctured disc.

By using linearizing coordinates we may assume that the foliation
$\widetilde{\mathcal{F}}|_{\widetilde{T}}$ extends to the set $$\widetilde{T}_\infty=\{(t,x):|t|=\sqrt{r_1r_2},x\in\mathbb{C}\}$$
and  is the suspension of a hyperbolic automorphism of $\mathbb{C}$. In this case $\tilde{\phi}^T(w)$ is defined for all ${w}\in {\widetilde{T}_\infty}$ and all $T\in\mathbb{R}$. Then   $f(\tilde{w}):=\tilde{\phi}^{\sigma(\tilde{w})}(\tilde{w})$ is well defined and it is contained in the Hopf fiber $$\widetilde{D}_\infty =\lbrace(\tilde{t}_0,x):x\in\mathbb{C}\rbrace$$ through $\tilde{w}_2$ for all $\tilde{w}\in\mathcal{A}_0$.

 Observe that $f:\mathcal{A}_0\rightarrow \widetilde{D}_\infty$ is holomorphic,
because it is a continuous leaf preserving map. Identifying $\mathcal{A}_0$
with $\mathbb{D}\backslash\{0\}$ and $\widetilde{D}_\infty$ with $\mathbb{C}$ we have by  Riemann Extension
Theorem that $f$ extends to a holomorphic map $f:\mathbb{D}\mapsto\mathbb{C}$,
$f(0)=0$. Then $f(\tilde{w})$ tends to $o=(\tilde{t}_0,0)$ as $\tilde{w}$ tends to the exceptional divisor. Since $\widetilde{\mathcal{F}}|_{\widetilde{T}_\infty}$ is the suspension of a hyperbolic
automorphism of $\mathbb{C}$, there exists a set $\mathfrak{T}\subset \widetilde{T}$
such that
\begin{enumerate}
\item $\mathfrak{T}$ contains all segment of orbit with endpoints in $\mathfrak{T}$
\item $\mathfrak{T}$ contains the set $$\lbrace(t,x):|t|=\sqrt{r_1r_2},|x|<\epsilon\rbrace$$
for some $\epsilon>0$.
\end{enumerate}
Since $f(0)=0$, we can set  $$D^*=\{(t_0,x):|x|<\delta_1\}$$ with $\delta_1>0$ such that $\mathcal{A}=\bar{h}(D^*)$  and $f(\mathcal{A})$ are contained in $\mathfrak{T}$.
The proof of the injectivity of $f|_{\mathcal{A}}$
given in \ref{sec redressing}  also works in this case and Proposition \ref{redressing} follows.

\section{Proof of Proposition \ref{geodesica transversal} and Lema \ref{lema empuja geodesica}}\label{potencial theory}
  Given a proper subdomain $D\subset\overline{\mathbb{C}}$  we denote by
$\omega_D(z)$ the harmonic measure with pole at $z\in D$. Recall that
 $\omega_D(z)$ is a probability measure on $\partial D$ and, fixed a Borel subset $B$ of $\partial D$, the function $z\mapsto \omega_D(z,B)$ is harmonic on $D$. We will use the following subordination principle.
 \begin{thm}\label{subordination principle}
 Let $D_1$ and $D_2$ be domains in $\overline{\mathbb{C}}$ with non-polar boundaries. Suppose that $D_1\subset D_2$ and let $B$ be  a Borel subset of $\partial D_1\cap\partial D_2$. Then $$\omega_{D_1}(z,B)\le \omega_{D_2}(z,B) \mbox{ for all }z\in D_1.$$ Moreover, if $\omega_{D_1}(z,B)>0$ the equality holds only if $D_1=D_2$.
 \end{thm}
 For the proof of this result and general background on harmonic measures we refer to \cite{ransford}.
\subsection{Proof of Lemma \ref{lema empuja geodesica}}
It is known that the geodesic $\tilde{\gamma}$ is given by the level set $$\omega_{\tilde{U}}(z,C)=\frac{1}{2}.$$
By Theorem \ref{subordination principle} we have $\omega_U(z,C)<\omega_{\widetilde{U}}(z,C)$ for all $z\in U$. Then for $z$ in the geodesic $\widetilde{\gamma}$  we have $\omega_U(z,C)<\frac 1 2$. But it is easy to see that $\omega_U(z,C)\ge \frac{1}{2}$ for all $z\in\overline{\Omega}\cap U$ and the lemma follows.
\subsection{Proof of Proposition \ref{geodesica transversal}}
Let $\widetilde{D}$ be the interior domain of $\widetilde{T}_{\mathfrak{w}}(R)$ and let  $\widetilde{C}=\widetilde{T}_{\mathfrak{w}}(\partial R^+)$ where $\partial R^+=\{\tau+\theta i\in\partial R: \theta\ge 0\}$. In $\widetilde{D}$, the geodesic $\Gamma_{\mathfrak{w}}$ is defined by the level set
$$\omega_{\widetilde{D}}(z,\widetilde{C})=\frac{1}{2}.$$

Provided $\theta_0$ is big enough there are real numbers ${\alpha_1},{\alpha_2}$ with $|{\alpha_1}|,|{\alpha_2}|\le \theta_0$ such that \footnote{Here $\vartheta_1$ is the function given by  Proposition \ref{creciente-decreciente}}:
\begin{enumerate}
\item $\widetilde{T}_{\mathfrak{w}}({\alpha_1}i)=\tau_1+\vartheta_1(0)i $, where $\tau_1$  is maximal wi
th $$ -2\ln L\le \tau_1\le 2\ln L.$$
\item $\widetilde{T}_{\mathfrak{w}}(\ln\frac{r_2}{r_1}+{\alpha_2}i)=\tau_2+\vartheta_1(0)i $, where $\tau_2$ is minimal with $$ \ln\frac{r_2}{r_1}-2\ln L\le \tau_2\le \ln\frac{r_2}{r_1}+2\ln L.$$
\end{enumerate}

 From  item (1) above we obtain $$\vartheta_1(0)=\tilde{\theta}(\mathfrak{w},{\alpha_1}i).$$ Then by Proposition
\ref{creciente-decreciente} we have
$$\vartheta_1(0)=\tilde{\theta}(\mathfrak{w},{\alpha_1}i)\ge \vartheta_1({\alpha_1}),$$ so
 ${\alpha_1}\le 0$ because $\vartheta_1$ is increasing. In the same way we have  ${\alpha_2}\le 0$.
 Then, if $C_1$ is the superior segment in $\partial R$ defined by the points ${\alpha_1}i$ and  $(\ln\frac{r_2}{r_1}+{\alpha_2}i)$, we have $\partial R^+\subset C_1$ and therefore $\widetilde{C}_1:=\widetilde{T}_{\mathfrak{w}}(C_1)$ contains $\widetilde{C}$. Therefore
 \begin{equation}
 \omega_{\widetilde{D}}(z,\widetilde{C})\le \omega_{\widetilde{D}}(z,\widetilde{C}_1)\mbox{ for all }z\in \widetilde{D}.\label{primera desigualdad}
 \end{equation}
 Let $\widetilde{D}_1$ be the region bounded by:
 \begin{enumerate}
 \item $\partial \widetilde{D}\backslash \widetilde{C}_1$
 \item \label{item2}$\{\tau+\vartheta_1 (0)i:\tau_1\le\tau\le 2\ln L\}$
 \item  \label{item3} $\{\tau+\vartheta_1 (0)i:\ln\frac{r_2}{r_1}-2\ln L\le\tau\le \tau_2\}$
 \item  \label{item4}$\{\tau+\theta i\in \partial \widetilde{R}_0: \theta\ge \vartheta_1 (0)\}$ \footnote{$\widetilde{R}_0$ is given in \ref{rectangulo pequeno}}.
 \end{enumerate}
 Clearly $\widetilde{D}_1\subset\widetilde{D}$ and $\partial \widetilde{D}\backslash \widetilde{C}_1\subset \partial \widetilde{D}_1\cap\partial \widetilde{D}$. Then by Theorem \ref{subordination principle} we have
\begin{equation}
 \omega_{\widetilde{D}_1}(z,\partial \widetilde{D}\backslash \widetilde{C}_1)\le \omega_{\widetilde{D}}(z,\partial \widetilde{D}\backslash \widetilde{C}_1)\mbox{ for all }z\in \widetilde{D}_1.\nonumber
 \end{equation} From this, if $\widetilde{C}_2$ is the subset of $\partial\widetilde{D}_1$ composed by the sets given in items \ref{item2}, \ref{item3} and \ref{item4} above we have that
 \begin{equation}
 \omega_{\widetilde{D}}(z,\widetilde{C}_1)\le
 \omega_{\widetilde{D}_1}(z, \widetilde{C}_2)\mbox{ for all }z\in \widetilde{D}_1.\label{segunda desigualdad}
 \end{equation}
 Let $\widetilde{D}_2$ be the connected component containing $z$ of the complement of $\widetilde{C}_2$ in the interior of $\widetilde{R}$. Clearly we have $\widetilde{D}_1\subset \widetilde{D}_2$ and $\widetilde{C}_2\subset \partial\widetilde{D}_1\cap \partial \widetilde{D}_2$. Then by Theorem \ref{subordination principle} we have
 \begin{equation}
 \omega_{\widetilde{D}_1}(z,\widetilde{C}_2)\le \omega_{\widetilde{D}_2}(z,\widetilde{C}_2)\mbox{ for all }z\in \widetilde{D}_1.\label{tercera desigualdad}
 \end{equation}
 Set $$\widetilde{C}_3=\{\tau+\theta i\in\partial\widetilde{R}: \theta\le\vartheta_1(0)\}$$ and let $\widetilde{D}_3$ be the region bounded by
  \begin{enumerate}
 \item $\widetilde{C}_2$
 \item $\widetilde{C}_3$
 \item $\{\tau+\vartheta_1 (0)i:-2\ln L\le\tau\le \tau_1\}$
 \item  $\{\tau+\vartheta_1 (0)i:\tau_2\le\tau\le\ln\frac{r_2}{r_1}+2\ln L\}$.
 \end{enumerate}
 Again by Theorem \ref{subordination principle} we have
 \begin{equation} \omega_{\widetilde{D}_3}(z,\widetilde{C}_3)\le \omega_{\widetilde{D}_2}(z,\widetilde{C}_3)\mbox{ for all }z\in \widetilde{D}_3,\nonumber
 \end{equation}
 so
\begin{equation} \omega_{\widetilde{D}_2}(z,\partial\widetilde{D}_2\backslash\widetilde{C}_3)\le \omega_{\widetilde{D}_3}(z,\partial\widetilde{D}_3\backslash\widetilde{C}_3)\mbox{ for all }z\in \widetilde{D}_3.\nonumber
 \end{equation} Then, since $\widetilde{C}_2\subset\partial\widetilde{D}_2\backslash\widetilde{C}_3$ we have
 \begin{equation} \omega_{\widetilde{D}_2}(z,\widetilde{C}_2)\le \omega_{\widetilde{D}_3}(z,\partial\widetilde{D}_3\backslash\widetilde{C}_3)\mbox{ for all }z\in \widetilde{D}_3.\label{cuarta desigualdad}
 \end{equation}
 Therefore, if $\widetilde{C}_4= \partial\widetilde{D}_3\backslash\widetilde{C}_3$ from equations \ref{primera desigualdad}, \ref{segunda desigualdad}, \ref{tercera desigualdad} and \ref{cuarta desigualdad} we obtain:
 \begin{equation} \omega_{\widetilde{D}}(z,\widetilde{C})\le \omega_{\widetilde{D}_3}(z,\widetilde{C}_4)\mbox{ for all }z\in \widetilde{D}\cap \widetilde{D}_3.\label{quinta desigualdad}
 \end{equation}
 Fixed $z\in \widetilde{D}\cap \widetilde{D}_3$, let $f:\widetilde{R}\rightarrow\overline{\mathbb{D}}$ be a homeomorphism such that:
 \begin{enumerate}
 \item $f$ maps  $\mathcal{R}=\mbox{int}(\widetilde{R})$  biholomorphically onto $\mathbb{D}$
 \item $f(z)=0$.
\end{enumerate}
Set $$\partial\widetilde{R}^+=\{\tau+\theta i\in\partial\widetilde{R}: \theta\ge 0\}.$$
We know that
 \begin{equation} \omega_{\widetilde{D}_3}(z,\widetilde{C}_4)= \omega_{f(\widetilde{D}_3)}(0,f(\widetilde{C}_4))\label{sexta desigualdad}
 \end{equation} and
 \begin{equation} \omega_{\mathcal{R}}(z,\partial\widetilde{R}^+)= \omega_{\mathbb{D}}(0,f(\partial\widetilde{R}^+)).\label{setima desigualdad}
 \end{equation}
 It is not difficult to see that, if $\ln\frac{r_2}{r_1}$ and $\widetilde{\theta}_0$ are  large, the Jordan curve $f(\partial\widetilde{D}_3)$ is  uniformly close to $\partial\mathbb{D}$ and the segment $f(\widetilde{C}_4)$ is uniformly close to $f(\partial\widetilde{R}^+)$.
 Thus, given $\epsilon>0$ and given a compact set $K\subset \mbox{int}(\widetilde{R}_0)$, it can be proved that we can take
 $\ln\frac{r_2}{r_1}$ and $\widetilde{\theta}_0$ large enough such that
 \footnote{Here we use the good dependence on the boundary of the uniformization  of a Jordan region (see \cite{pommerenke} p.26) and the invariance of  harmonic measures by conformal maps. }:
 \begin{equation} \omega_{f(\widetilde{D}_3)}(0,f(\widetilde{C}_4))\le\omega_{\mathbb{D}}(0,f(\partial\widetilde{R}^+))+\epsilon \nonumber
 \end{equation} for all $z\in K$
  and therefore by \ref{quinta desigualdad}, \ref{sexta desigualdad} and \ref{setima desigualdad}:
 \begin{equation} \omega_{\widetilde{D}}(z,\widetilde{C})\le\omega_{\mathcal{R}}(z,\partial\widetilde{R}^+)+\epsilon \nonumber
 \end{equation} for all $z\in K$.
 By similar arguments we can obtain
  \begin{equation} \omega_{\widetilde{D}}(z,\widetilde{C})\ge\omega_{\mathcal{R}}(z,\partial\widetilde{R}^+)-\epsilon \mbox{ for all }z\in K.\nonumber
 \end{equation}
 Thus, when $\ln\frac{r_2}{r_1}$ and $\widetilde{\theta}_0$ tends to $+\infty$ we have that the function $ z\mapsto\omega_{\widetilde{D}}(z,\widetilde{C})$ converge to the function $ z\mapsto\omega_{\mathcal{R}}(z,\partial\widetilde{R}^+)$  uniformly on the compacts subsets of $\mbox{int}(\widetilde{R}_0)$ and uniformly on $\mathfrak{w}\in\partial V_1$. So, fixed a compact set $K\subset \mbox{int}(\widetilde{R}_0)$, when $\ln\frac{r_2}{r_1}$ and $\widetilde{\theta}_0$ tends to $+\infty$ we have that, uniformly on $\mathfrak{w}\in\partial V_1$, the level set $$\{z\in K:\omega_{\widetilde{D}}(z,\widetilde{C})=\frac{1}{2}\}$$ tends in the $C^\infty$ topology to the level set $$\{z\in K:\omega_{\mathcal{R}}(z,\partial\widetilde{R}^+)=\frac{1}{2}\}=\{\tau+\theta i\in K:\theta=0\}$$ and the proposition follows.

\section{Projective holonomy representation}\label{projective holonomy representation}
Given a non-1-dicritical foliation $\mathcal{F}$, we define the projective holonomy representation of $\mathcal{F}$ as follows. Let $E$ be the exceptional divisor after a blow up and denote also by $\mathcal{F}$ the strict transform of $\mathcal{F}$.
Clearly
 $E^{*}=E\backslash\mathcal{\textrm{Sing$({\mathcal{F}})$ }}$
is a leaf of ${\mathcal{F}}.$ Let $q$ be a point in $E^{*}$
and $\Sigma$ a small complex disc passing through $q$ and transverse
to ${\mathcal{F}}$. For any loop $\gamma$ in $E^{*}$
based on $q$ there is a holonomy map $H_{\mathcal{F}}(\gamma):(\Sigma,q)\mapsto(\Sigma,q)$
which only depends on the homotopy class of $\gamma$ in the fundamental
group $\Gamma=\pi_{1}(E^{*})$. The map $H_{\mathcal{F}}:\Gamma\mapsto\textrm{Diff}(\Sigma,q)$
is known as the projective holonomy representation of $\mathcal{F}$.
Identifying $(\Sigma,q)\simeq(\mathbb{C},0)$ the image of $H_{\mathcal{F}}$
defines up to conjugation a subgroup of $\textrm{Diff$(\mathbb{C},0)$}$
which is known as the holonomy group of $\mathcal{F}$.

The representations $H:\Gamma\mapsto\textrm{Diff}(\mathbb{C},0)$
and $H':\Gamma'\mapsto\textrm{Diff}(\mathbb{C},0)$ are topologically
conjugated if there exist an isomorphism $\varphi:\Gamma\mapsto\Gamma'$
and a germ of homeomorphism $h:(\mathbb{C},0)\mapsto(\mathbb{C},0)$
such that $H'\circ\varphi(\gamma)=h\circ H(\gamma)\circ h^{-1}$ for
all $\gamma\in\Gamma$.

\subsection{Proof of Theorem \ref{second result}}\label{pt2}

 Let $p_1$,...,$p_k$ be the singularities of $\mathcal{F}$ in $E$. Let $S_1$,...,$S_k$ be irreducible separatrices through  $p_1$,...,$p_k$ respectively all different from $E$. Let $\widetilde{S}_j=h(S_j)$ and let $\tilde{p}_j$ be the point where $\widetilde{S}_j$ meets the exceptional divisor $E$.  Let $(t,x)$ be coordinates in $\widehat{\mathbb{C}^2}$ such that  $\pi(t,x)=(x,tx)\in\mathbb{C}^2$. We can assume that the points $p_j$ and $\tilde{p}_j$ are contained in this coordinates, that is, $p_j=(t_j,0)$ and $\tilde{p}_j=(\tilde{t}_j,0)$ with $t_j,\tilde{t}_j\in\mathbb{C}$.   We can perform the constructions of the proof of Theorem \ref{main result} for each $S_j$  to obtain:
\begin{enumerate}
 \item Another topological equivalence $\bar{h}$ between $\mathcal{F}$ and $\widetilde{\mathcal{F}}$
 \item Some constants  $r_j,\tilde{r}_j>0$
 \end{enumerate}
 such that  the following properties hold:
\begin{enumerate}
\item The sets $\{|t-{t}_j|\le r_j\}$   are pairwise disjoint
\item The sets $\{|t-\tilde{t}_j|\le \widetilde{r}_j\}$   are pairwise disjoint and they does not contain singularities of $\widetilde{\mathcal{F}}$ other than the $\tilde{p}_j$
\item For each $j$ and any $w\in\mbox{dom}(\bar{h})$ we have that: $$w\in\{|t-{t}_j|< r_j\}\mbox{ if and only if } \bar{h}(w)\in\{|t-\tilde{t}_j|< \tilde{r}_j\}$$
$$w\in\{|t-{t}_j|> r_j\}\mbox{ if and only if } \bar{h}(w)\in\{|t-\tilde{t}_j|> \tilde{r}_j\}$$
$$w\in\{|t-{t}_j|= r_j\}\mbox{ if and only if } \bar{h}(w)\in\{|t-\tilde{t}_j|= \tilde{r}_j\}$$
\item The Hopf fibers in $\{|t-{t}_j|= r_j\}$ are mapped by $\bar{h}$ into Hopf fibers in $\{|t-\tilde{t}_j|= \tilde{r}_j\}$.

\end{enumerate}
Let $W$ and $\widetilde{W}$ be the complement in $\widehat{\mathbb{C}^2}$ of the sets $$\bigcup_{j=1}^k\{|t-{t}_j|< r_j\}$$ and $$\bigcup_{j=1}^k\{|t-\tilde{t}_j|< \tilde{r}_j\}$$ respectively.
Let $\widetilde{\Sigma}$ be a local Hopf fiber in $\{|t-\tilde{t}_1|= \tilde{r}_1\}$ passing through a point $\tilde{q}\in E$ such that $\bar{h}^{-1}(\widetilde{\Sigma})$ is contained in a local Hopf fiber $\Sigma$ in $\{|t-{t}_1|={r}_1\}$. The local homeomorphism $\bar{h}^{-1}:\widetilde{\Sigma}\rightarrow\Sigma$ will be the homeomorphism associated to the desired projective holonomy conjugacy. In order to construct this projective holonomy conjugacy  we must first define an isomorphism $$\varphi:\pi_1(E\backslash\mbox{Sing}(\widetilde{\mathcal{F}}),\tilde{q})\rightarrow\pi_1(E\backslash\mbox{Sing}(\mathcal{F}),q).$$
As a first step we proceed to define a homomorphism $$\varphi:\pi_1(E\backslash\mbox{Sing}(\widetilde{\mathcal{F}}),\tilde{q})\rightarrow\pi_1(E\backslash\mbox{Sing}(\mathcal{F}),q).$$
In what follows we use the same letter $\rho$ to denote the projections $W\rightarrow E\cap W$ and $\widetilde{W}\rightarrow E\cap\widetilde{W}$ associated to the Hopf fibration.
Let $\tilde{\gamma}$ be a loop in $E\cap \widetilde{W}\backslash\mbox{Sing}(\widetilde{\mathcal{F}})$ based on $\tilde{q}$. Let $\tilde{\gamma}_1\subset\widetilde{W}$ be any loop based on a point $\tilde{q}_1\in\widetilde{\Sigma}\backslash\{\tilde{q}\}$, disjoint of the exceptional divisor $E$  and such that $\tilde{\gamma}_1$ is freely homotopic to $\tilde{\gamma}$ in $\widetilde{W}$. Let $\gamma\subset{{E}\cap W}$ be the loop based at $q=\Sigma\cap E$ given by $\gamma=\rho(\bar{h}^{-1}(\tilde{\gamma}_1))$.\\

\noindent\emph{Assertion 1.  The loop $\gamma$ defines an element in  $\pi_1(E\backslash\mbox{Sing}(\mathcal{F}),q)$ which depends only on the class of $\tilde{\gamma}$ in $\pi_1(E\backslash\mbox{Sing}(\widetilde{\mathcal{F}}),\tilde{q})$. The map  $\tilde{\gamma}\mapsto {\gamma}$ defines a homomorphism $\varphi$ between  $\pi_1(E\backslash\mbox{Sing}(\widetilde{\mathcal{F}}),\tilde{q})$ and $\pi_1(E\backslash\mbox{Sing}(\mathcal{F}),q)$.  Moreover, if $\widetilde{\mathcal{F}}$ does not have singularities other than the $\tilde{p}_j$ we have that $\varphi$ is an isomorphism.}\\

Let $\tilde{\gamma}'$ be another loop in $E\cap \widetilde{W}\backslash\mbox{Sing}(\widetilde{\mathcal{F}})$, contained in the same class of $\tilde{\gamma}$ in $$\pi_1(E\backslash\mbox{Sing}(\widetilde{\mathcal{F}}),\tilde{q}).$$ Then it is easy to see that \begin{equation}\label{ecu10}\tilde{\gamma}'=\tilde{\gamma}\mbox{ in }\pi_1(E\cap \widetilde{W},\tilde{q}).\end{equation} Let $\tilde{\gamma}'_1\subset\widetilde{W}$ be any loop based on a point $\tilde{q}'_1\in\widetilde{\Sigma}\backslash\{\tilde{q}\}$, disjoint of the exceptional divisor $E$  and such that $\tilde{\gamma}'_1$ is freely homotopic to $\tilde{\gamma}'$ in $\widetilde{W}$. We must show that $\gamma':=\rho(\bar{h}^{-1}(\tilde{\gamma}'_1))$ and $\gamma$ define the same class in $\pi_1(E\backslash\mbox{Sing}({\mathcal{F}}),{q})$. If $\tilde{\alpha}$ is a path connecting $\tilde{q}_1$ with $\tilde{q}'_1$ in $\widetilde{\Sigma}\backslash\{\tilde{q}\}$, we have that $\tilde{\alpha}\tilde{\gamma}'_1\tilde{\alpha}^{-1}$ is also freely homotopic to   $\tilde{\gamma}'$  in $\widetilde{W}$. Then, from equation \ref{ecu10} and the definition of $\tilde{\gamma}_1$ we obtain that $\tilde{\alpha}\tilde{\gamma}'_1\tilde{\alpha}^{-1}$ is homotopic with fixed base point to $\tilde{\gamma}_1$  in $\widetilde{W}$, that is  \begin{equation}\label{ecu11}(\tilde{\alpha}\tilde{\gamma}'_1\tilde{\alpha}^{-1})\tilde{\gamma}_1^{-1}=0\mbox{ in }\pi_1(\widetilde{W},\tilde{q}_1).\end{equation} Therefore
 \begin{equation}\label{ecu12}\rho(\tilde{\alpha}\tilde{\gamma}'_1\tilde{\alpha}^{-1}\tilde{\gamma}_1^{-1})=0\mbox{ in }\pi_1(E\cap\widetilde{W},\tilde{q}).\end{equation} Let $\rho^*$ denote the homomorphism $$\pi_1(\widetilde{W}\backslash E,\tilde{q}_1)\rightarrow\pi_1(E\cap\widetilde{W},\tilde{q})$$ induced by $\rho$ and let $\tilde{\sigma}\subset\widetilde{\Sigma}\backslash\{\tilde{q}\}$ be a simple loop based on $\tilde{q}_1$ and  surrounding $\tilde{q}$. By standard algebraic topological arguments we can see that the kernel of $\rho^*$ is generated by the loop $\tilde{\sigma}$. Then, from equation \ref{ecu12} we obtain that
 \begin{equation}\label{ecu13} \tilde{\alpha}\tilde{\gamma}'_1\tilde{\alpha}^{-1}\tilde{\gamma}_1^{-1}=k\sigma\mbox{ in }\pi_1(\widetilde{W}\backslash E,\tilde{q}_1),\mbox{ for some }k\in\mathbb{Z}.\end{equation} Then, since $\bar{h}^{-1}$ maps $\widetilde{W}\backslash E$ into $W\backslash E$:
 \begin{equation}\label{ecu14} \bar{h}^{-1}(\tilde{\alpha}\tilde{\gamma}'_1\tilde{\alpha}^{-1}\tilde{\gamma}_1^{-1})=k\bar{h}^{-1}(\sigma)\mbox{ in }\pi_1({W}\backslash E,{q}_1),\mbox{ where }q_1=\bar{h}^{-1}(\tilde{q}_1)\in\Sigma.\end{equation} From this equation, since $\rho(\bar{h}^{-1}(\tilde{\sigma}))$ is a constant path, we obtain
 \begin{equation}\label{ecu15} \rho(\bar{h}^{-1}(\tilde{\alpha}\tilde{\gamma}'_1\tilde{\alpha}^{-1}\tilde{\gamma}_1^{-1}))=0\mbox{ in }\pi_1(E\cap{W},q).\end{equation} Since  $\rho(\bar{h}^{-1}(\tilde{\alpha}))$ is contained in $\Sigma$ it is easy to see that
 $$\rho(\bar{h}^{-1}(\tilde{\alpha}\tilde{\gamma}'_1\tilde{\alpha}^{-1}))=\rho(\bar{h}^{-1}(\tilde{\gamma}'_1))\mbox{ in }\pi_1(E\cap{W},q).$$
 Then, from equation \ref{ecu15} we obtain
 \begin{equation} \label{ecu16} \rho(\bar{h}^{-1}(\tilde{\gamma}'_1))\rho(\bar{h}^{-1}(\tilde{\gamma}_1^{-1}))=0\mbox{ in }\pi_1(E\cap{W},q),\end{equation}
  hence $\gamma'$  and $\gamma$ define the same class in $\pi_1(E\backslash \mbox{Sing}(\widetilde{\mathcal{F})},q)$ and therefore the map $$\varphi:\pi_1(E\backslash\mbox{Sing}(\widetilde{\mathcal{F}}),\tilde{q})\rightarrow\pi_1(E\backslash\mbox{Sing}(\mathcal{F}),q)$$ is well defined. It is easy to prove that $\varphi$ is a homomorphism. If $\widetilde{\mathcal{F}}$ does not have singularities other than the $\tilde{p}_j$, then we can perform the constructions above in the opposite direction to construct the inverse of $\varphi$, so the proof of Assertion 1 is complete.\\

\noindent\emph{Assertion 2.  Let $\tilde{\gamma}$ be a loop in $E\cap \widetilde{W}\backslash\mbox{Sing}(\widetilde{\mathcal{F}})$ based on $\tilde{q}$ and let $\tilde{f}:(\widetilde{\Sigma},\tilde{q})\rightarrow (\widetilde{\Sigma},\tilde{q})$ be the holonomy associated to $\tilde{\gamma}$. This holonomy map induces  the map $$f=\bar{h}^{-1}\tilde{f}\bar{h}:(\Sigma,q)\rightarrow(\Sigma,q).$$
Then $f$ coincides with the holonomy map associated to $\varphi(\tilde{\gamma})$.}\\

Let $w\in\Sigma\backslash\{q\}$ and consider the path $\alpha=\bar{h}^{-1}(\tilde{\alpha})$, where $\tilde{\alpha}$ is the lift\footnote{by the Hopf fibration} to the leaf through $h(w)$  of the loop $\widetilde{\gamma}$. Clearly the path $\alpha$ is contained in a leaf and connects $w$ with $f(w)$.
We know that the set ${E}\cap W$ can be retracted into a subset  $B$ which is a bouquet of $(k-1)$ circles based on $q$. By using the Hopf Fibration we can "lift" this retraction to the leaves close to the exceptional divisor. Thus, if $w$ is close enough to $q$, the path $\alpha$ is homotopic in the leaf through $w$ to a path  $\beta$ contained in the Hopf fibration restricted to $B$. Clearly the path $\beta$ connects $w$ with $f(w)$. Let $\beta_0$ be the projection of $\beta$ on $B$. Clearly  $\beta_0=\varphi(\tilde{\gamma})$ in $\pi_1(E\backslash\mbox{Sing}(\mathcal{F}),q)$. Let $\gamma$ be a geodesic representant of
$\beta_0$ in $\pi_i(B,q)$. The homotopy between $\beta_0$ and $\gamma$ can be performed in the image of $\beta_0$. Let $\gamma_w$ be the lift of $\gamma$ to the leaf through $w$. Then it is not difficult to see  that $\beta$ is homotopic to $\gamma_w$ by a homotopy contained in the image of $\beta$. Therefore $\beta$ is homotopic in the leaf through $w$ to the path $\gamma_w$ and in particular we have that $f(w)$ is the endpoint of $\gamma_w$. But the endpoint of $\gamma_w$ is just the image of $w$ by the holonomy map associated to $\varphi(\tilde{\gamma})$. Thus Assertion 2 is proved.

In view of  the assertions above, to complete the proof of Theorem \ref{second result} it suffices to show that  $\widetilde{\mathcal{F}}$ does not have singularities other than the $\tilde{p}_j$.
 Suppose that there exists a singularity $\tilde{p}$ of $\widetilde{\mathcal{F}}$ different from all $\tilde{p}_j$. Let $\tilde{\gamma}$ be a loop based on $\tilde{q}$ of type $\tilde{\gamma}=\alpha\epsilon\alpha^{-1}$ where $\epsilon$ is a loop surrounding closely the singularity $\tilde{p}$. It is easy to see in this case  that $\varphi(\tilde{\gamma})=0$ in $\pi_1(E\backslash\mbox{Sing}(\mathcal{F}),q)$. Therefore, it follows from the assertion above  that the holonomy map associated to $\tilde{\gamma}$ is the identity. Let $D\subset E$ be a small disc such that $\overline{D}\cap\mbox{Sing}(\widetilde{\mathcal{F}})=\{\tilde{p}\}$. Then the circle $\partial D$ lift to a circle $\delta$ contained in a leaf close to $\partial D$. The circle $\bar{h}^{-1}(\delta)$ is contained in a leaf $L$ and its projection in ${E}\cap W$ is null-homotopic. As in the proof of the assertion, provided $\delta$ is close enough to $\partial D$ we have that $\bar{h}^{-1}(\delta)$ is null homotopic in $L\cap W$. Then $\bar{h}^{-1}(\delta)$ is the boundary of a simply connected domain  in $L$. Therefore $\delta$ is the boundary of a simply connected domain $\Omega$ in the leaf containing $\delta$. Moreover we have $\Omega\subset\widetilde{W}$, so we have that $\Omega$ is a disc close to $D$ as $\delta$ is close to $\partial D$, but this property is only possible if $\tilde{p}$ is a regular point.

 \begin{rem}As was mentioned in Remark \ref{aaa}, the isomorphism $\varphi$ is in fact a geometric one, that is, $\varphi$ is induced by a homeomorphism.  This is basically a direct application of the main result of  \cite{MM5}.
 \end{rem}

\end{document}